\documentclass{amsart}

\usepackage{mathrsfs}
\usepackage{kbordermatrix}
\usepackage{enumerate}
\usepackage[section]{algorithm}
\usepackage{algorithmic}
\usepackage{multirow}
\usepackage{xcolor} 

\usepackage{graphicx}
\usepackage{amsmath,amsfonts,amsbsy,amssymb}
\usepackage{fixmath}
\usepackage{amssymb,latexsym}

\def\bu{\pmb{u}}

\def\bx{\pmb{x}}
\def\by{\pmb{y}}
\def\bz{\pmb{z}}

\def\bone{\pmb{1}}

\def\bbC{\mathbb{C}}
\def\bbO{\mathbb{O}}
\def\bbP{\mathbb{P}}
\def\bbR{\mathbb{R}}
\def\bbX{\mathbb{X}}

\def\scrG{\mathscr{G}}

\def\cI{\mathcal{I}}

\def\cP{\mathcal{P}}
\def\cQ{\mathcal{Q}}
\def\cR{\mathcal{R}}
\def\cX{\mathcal{X}}
\def\cY{\mathcal{Y}}

\def\wtd{\widetilde}
\def\what{\widehat}

\DeclareMathOperator*{\opt}{opt}
\DeclareMathOperator{\diag}{diag}
\DeclareMathOperator{\dist}{dist}
\DeclareMathOperator{\eig}{eig}
\DeclareMathOperator{\grad}{grad}
\DeclareMathOperator{\rank}{rank}
\DeclareMathOperator{\sym}{sym}
\DeclareMathOperator{\tr}{trace} 
\DeclareMathOperator{\F}{F}
\DeclareMathOperator{\HH}{H}
\DeclareMathOperator{\T}{T}

\def\tol{{\tt tol}}

\newtheorem{theorem}{Theorem}[section]
\newtheorem{lemma}{Lemma}[section]

\newtheorem{corollary}{Corollary}[section]
\newtheorem{remark}{{\sc Remark}}[section]

\allowdisplaybreaks

\numberwithin{equation}{section}

\def\sss{\scriptstyle}

\title{Trace Ratio Optimization with an Application to Multi-view Learning}
\author[L. Wang] {Li Wang}
\author[L.-H. Zhang]{Lei-Hong Zhang}
\author[R.-C. Li]{Ren-Cang Li}
\thanks{
	Li Wang is with
	Department of Mathematics and Department of Computer Science and Engineering, University of Texas at Arlington, Arlington, TX 76019-0408, USA. 
	E-mail: li.wang@uta.edu.  \\
	\indent Lei-Hong Zhang is with
	School of Mathematical Sciences and Institute of Computational Science, Soochow University, Suzhou 215006, Jiangsu, China. 
	E-mail: longzlh@suda.edu.cn. \\
	\indent Ren-Cang Li is with
	Department of Mathematics, University of Texas at Arlington, Arlington, TX 76019-0408, USA. 
	E-mail: rcli@uta.edu.}

\begin{document}
\maketitle

\begin{abstract}
A trace ratio optimization problem over the Stiefel manifold is  investigated from the perspectives of both theory
and numerical computations. At least three special cases of the problem have arisen from Fisher's linear discriminant analysis,
canonical correlation analysis, and unbalanced Procrustes problem, respectively. Necessary conditions in the form of
nonlinear eigenvalue problem with eigenvector dependency are established and a numerical method based on
the self-consistent field (SCF) iteration is designed and proved to be always convergent. As an application to  multi-view subspace learning,
a new framework and its instantiated concrete models are proposed and demonstrated on real world data sets. Numerical results
show that the efficiency of the proposed numerical methods and effectiveness of the new multi-view subspace learning models.
\end{abstract}

\section{Introduction}
We are concerned with the following trace ratio maximization problem
\begin{subequations}\label{eq:main-op}
	\begin{equation}\label{eq:main-op-1}
	\max_{X^{\T}X=I_k}f_{\theta}(X),
	\end{equation}
	where $1\le k<n$, $I_k$ is the $k\times k$ identity matrix, and
	\begin{equation}\label{eq:main-f}
	f_{\theta}(X)=\frac {\tr(X^{\T}AX+X^{\T}D)}{[\tr(X^{\T}BX)]^{\theta}},
	\end{equation}
\end{subequations}
$A,\,B\in\bbR^{n\times n}$ are symmetric and $B$ is positive semi-definite with $\rank(B)>n-k$, $D\in\bbR^{n\times k}$,  matrix variable $X\in\bbR^{n\times k}$, and parameter $0\le\theta\le 1$. The condition that $\rank(B)>n-k$ ensures the denominator
of $f_{\theta}(X)$ is always positive for any $X$ such that $X^{\T}X=I_k$.

Problem \eqref{eq:main-op} is a maximization problem on the Stiefel manifold \cite{abms:2008}:
$$
\bbO^{n\times k}=\{X\in\bbR^{n\times k}\,:\,X^{\T}X=I_k\}.
$$
A seemingly more general case than \eqref{eq:main-op} is to maximize
\begin{equation}\label{eq:main-op:gl}
\frac {\tr(X^{\T}AX+X^{\T}D)+c}{[\tr(X^{\T}BX)]^{\theta}}
=\frac {\tr(X^{\T}[A+c/k)I_n]X+X^{\T}D)}{[\tr(X^{\T}BX)]^{\theta}}
\end{equation}
over $X\in\bbO^{n\times k}$, but it is actually a special case of \eqref{eq:main-op} due to the
reformulation on the right-hand side of \eqref{eq:main-op:gl}, where $c\in\bbR$ is a scalar constant.

\subsection{Previous Work}
There are three special cases arising  from data science that have been
fairly well studied in the last decades.

The first special case is with $D=0$ and $\theta=1$:
\begin{equation}\label{eq:OLDA}
\max_{X\in\bbO^{n\times k}}\frac {\tr(X^{\T}AX)}{\tr(X^{\T}BX)}
\end{equation}
arising from Fisher's linear discriminant analysis
\cite{ngbs:2010,zhln:2010,zhln:2013} in the setting of supervised learning. In \cite{ngbs:2010}, \eqref{eq:OLDA} is converted into a zero-finding problem
$$
\mbox{solve}\,\,\eta(\rho)=0
\quad\mbox{with}\quad
\eta(\rho):=\max_{X\in\bbO^{n\times k}}\tr(X^{\T}(A-\rho B)X).
$$
It is proved there that $\eta(\rho)$ is a non-increasing function of $\rho$ and has a unique zero in the generic case, and
the Newton method is applied to find the zero \cite{ngbs:2010}.
This idea of such conversion can be traced back to \cite{dink:1967} in 1967 for fractional programming. In
\cite{zhln:2010,zhln:2013}, however, \eqref{eq:OLDA} is treated as a maximization problem on the Stiefel manifold $\bbO^{n\times k}$.
Their major results are as follows.
The KKT condition of \eqref{eq:OLDA} with respect to the Stiefel manifold is a nonlinear eigenvalue problem
with eigenvector dependency (NEPv) \cite{cazb:2018}:
\begin{equation}\label{eq:NEPv}
H(X)X=X\Lambda,
\end{equation}
where $H(X)\in\bbR^{n\times n}$ is a matrix-valued function of $X\in\bbO^{n\times k}$ and symmetric,
and $\Lambda\in\bbR^{k\times k}$ is not really an unknown since it can be expressed as $\Lambda=X^{\T}H(X)X\in\bbR^{k\times k}$.
NEPv~\eqref{eq:NEPv} is then solved by the so-called self-consistent field (SCF) iteration
\begin{equation}\label{eq:SCF}
\mbox{iterate $H(X_{i-1})X_i=X_i\Lambda_{i-1}$, given $X_0$}
\end{equation}
where $X_i$ is an orthonormal basis matrix associated with the $k$ largest eigenvalues of $H(X_{i-1})$. It is noted that
SCF has been  widely
used in electronic structure calculations for decades \cite{mart:2004,sacs:2010}.
Remarkably, it is proved in \cite{zhln:2010,zhln:2013} that (i) problem \eqref{eq:OLDA} has no local maximizer but global ones, (ii) the SCF iteration
is monotonically convergent in terms of the objective value, and the iterates $X_i$ globally converge to a maximizer in the metric
on the Grassmann manifold $\scrG_k(\bbR^n)$ (the collection of all $k$-dimensional subspaces of $\bbR^n$),
and (iii) the convergence is locally quadratic in the generic case.
A shorter proof of the local quadratic convergence
can be found in \cite{cazb:2018}.

The second special case is with $A=0$ and $\theta=1/2$:
\begin{equation}\label{eq:OCCA-sub}
\max_{X\in\bbO^{n\times k}}\frac {\tr(X^{\T}D)}{\sqrt{\tr(X^{\T}BX)}},
\end{equation}
arising from orthogonal canonical correlation analysis (OCCA) \cite{zhwb:2020} as the kernel of an alternating
iterative scheme. The major results established in \cite{zhwb:2020} for \eqref{eq:OCCA-sub} can be summarized as follows.
The KKT condition of \eqref{eq:OCCA-sub} with respect to the Stiefel manifold $\bbO^{n\times k}$ does not exactly take the form
of an NEPv but can be equivalently converted into one like \eqref{eq:NEPv}. The latter can again be solved
by SCF iteration \eqref{eq:SCF} followed by a post-processing on $X_i$.
The method is
monotonically convergent in the objective value and the iterates $X_i$ globally converge to a critical point that satisfies
an established necessary condition  for a global minimizer.

The third special case is with $B=I_n$ and $\theta=0$:
\begin{equation}\label{eq:procrustes}
\max_{X\in\bbO^{n\times k}} {\tr(X^{\T}AX+X^{\T}D)}.
\end{equation}
This is a rather fundamental problem in numerical linear algebra, optimization, and applied statistics, among others, and is
commonly referred to as the unbalanced Procrustes problem \cite{chtr:2001,edas:1999,elpa:1999,godi:2004,huca:1962,zhys:2020,zhdu:2006}.
In particular,
the following least-squared minimization
\begin{equation}\label{eq:OPLS}
\min_{X\in {\mathbb O}^{n\times k}}\|CX-B\|_{\F}^2
\end{equation}
can be reformulated into \eqref{eq:procrustes} with $A=-C^{\T}C$ and $D=C^{\T}B/2$.
Both \eqref{eq:procrustes} and \eqref{eq:OPLS} can be found in many real world applications
including
the orthogonal least squares regression (OLSR)
for feature extraction \cite{zhwn:2016,nizl:2017}, the  multidimensional similarity structure analysis (SSA) \cite[Chapter 19]{boli:1987},
and the MAXBET problem \cite{geer:1984,liww:2015} in the canonical correlation analysis.
A recent work \cite{zhys:2020} turns \eqref{eq:OPLS}  into an equivalent NEPv
that  is solved by an SCF iteration in the form \eqref{eq:SCF}, among others.

There are some fundamental differences among these special cases \eqref{eq:OLDA}, \eqref{eq:OCCA-sub}, and
\eqref{eq:procrustes}, and their associated NEPv \cite{zhln:2010,zhln:2013,zhwb:2020,zhys:2020}.
Three major differences are as follows.
\begin{itemize}
	\item If $X_{\opt}$ is a maximizer of \eqref{eq:OLDA} then so is $X_{\opt}Q$ for any $Q\in\bbO^{k\times k}$. In this sense any maximizer of \eqref{eq:OLDA} is a representative of a class of maximizers. But that is not the case for \eqref{eq:OCCA-sub}, namely for a maximizer $X_{\opt}$ of \eqref{eq:OCCA-sub}, $X_{\opt}Q$ for a generic $Q\in\bbR^{k\times k}$
	is most likely not another maximizer.
	\item Problem \eqref{eq:OLDA} has no local maximizer but global ones, whereas \eqref{eq:OCCA-sub} and \eqref{eq:procrustes}
	are numerically demonstrated
	to have local  maximizers.
	\item NEPv \eqref{eq:NEPv} for \eqref{eq:OLDA} admits an orthogonal-invariance property: $H(XQ)=H(X)$
	for any  $Q\in\bbO^{k\times k}$, while the ones associated with \eqref{eq:OCCA-sub} and \eqref{eq:procrustes}
	do not have this property.
\end{itemize}

\subsection{Contributions and the Organization of This Paper}
In this paper, our goal is to thoroughly investigate problem \eqref{eq:main-op} as a maximization problem on the Stiefel manifold
$\bbO^{n\times k}$ in both theory and numerical computation. Our major contributions are as follows:
1) We turn the KKT condition of \eqref{eq:main-op} with respect to the Stiefel manifold  equivalently into
an NEPv; 2) We establish crucial necessary conditions, beyond the KKT condition, of local and global maximizers
in terms of the extreme eigenvalues of the NEPv; 3) We completely characterize the role of $D$ in how precisely it determines
maximizers, which is important because when $D=0$, any maximizer represents a class of many associated with an element of
the Grassmann manifold $\scrG_k(\bbR^n)$; 4) A numerical method based on SCF for the NEPv and a post-processing are proposed to efficiently solve \eqref{eq:main-op} as a consequence of our theoretical results, and the method is always
convergent; 5) As an application, we establish a new orthogonal multi-view subspace learning framework and solve it alternatingly
with our method for \eqref{eq:main-op} serving as the computational horse. The framework generalizes
the MAXBET problem.

The rest of this paper is organized as follows. In section~\ref{sec:KKT}, we derive the KKT condition, its associated NEPv, and
important theoretical issues to lay the foundation for the rest of the paper. In section~\ref{sec:RoleD}, we investigate
the role of $D$ in pining down the maximizers. In section~\ref{sec:SCF}, we propose our SCF method for problem \eqref{eq:main-op} and
conduct a detailed convergence analysis of the method. An application to multi-view subspace learning is carried out in
section~\ref{sec:MvSL}. Results of numerical experiments are reported in section~\ref{sec:NumExp}. Finally, we draw our conclusions in
section~\ref{sec:concl}.

{\bf Notation.}
$\bbR^{m\times n}$ is the set of $m\times n$ real matrices and $\bbR^n=\bbR^{n\times 1}$. $I_n\in\bbR^{n\times n}$ is
the identity matrix, and $\bone_n\in\bbR^n$ is the vector of all ones. $\|\bx\|_2$ is the 2-norm of vector $\bx\in \bbR^n$. For $B\in\bbR^{m\times n}$,
$\cR(B)$ is the column subspace and its singular values are denoted by
$\sigma_i(B)$ for $i=1,\ldots,\min(m,n)$ arranged in the nonincreasing order,
and
$$
\|B\|_2=\sigma_1(B),\,\,
\|B\|_{\F}=\sqrt{\sum_{i=1}^{{\rm rank}(B)}[\sigma_i(B)]^2},\,\,
\|B\|_{\tr}=\sum_{i=1}^{{\rm rank}(B)}\sigma_i(B)
$$
are the spectral norm and the Frobenius norm, and the trace norm (also known as the nuclear norm) of $B$,
respectively.
For a symmetric $B\in\bbR^{n\times n}$, $\eig(B)=\{\lambda_i(B)\}_{i=1}^n$ denotes the set of its eigenvalues (counted by multiplicities)
arranged in the nonincreasing order;
$B\succ 0 ~(\succeq 0)$ means that $B$ is positive definite (semi-definite).
MATLAB-like notation is used to access the entries of a matrix or vector:
$X_{(i:j,k:l)}$ to denote the submatrix of a matrix $X$, consisting of the intersections of
rows $i$ to $j$ and columns $k$ to $l$, and when $i : j$ is replaced by $:$, it means all rows, similarly for columns;
$v_{(k)}$ refers the $k$th entry of a vector $v$ and $v_{(i:j)}$
is the subvector of $v$ consisting of the $i$th to $j$th entries inclusive.

\section{KKT Condition and Associated NEPv}\label{sec:KKT}
For convenience, write $f_{\theta}(X)=g_{\theta}(X)+h_{\theta}(X)$, where
$$
g_{\theta}(X)=\frac {\tr(X^{\T}AX)}{[\tr(X^{\T}BX)]^{\theta}}, \quad
h_{\theta}(X)=\frac {\tr(X^{\T}D)}{[\tr(X^{\T}BX)]^{\theta}}.
$$
To find the KKT condition of  problem \eqref{eq:main-op} on the Stiefel manifold
$\bbO^{n\times k}$, first
we will need to find the gradient of $f_{\theta}$ on the manifold. We have
\begin{align*}
\frac{\partial f_{\theta}(X)}{\partial X}
&=\frac 2{[\tr(X^{\T}BX)]^{\theta}}\Big[\,A-\theta g_1(X) B\Big]\,X
+\frac 1{[\tr(X^{\T}BX)]^{\theta}}\Big[\,D-2\theta h_1(X) BX\Big].
\end{align*}
Let $\Pi_X(Z):=Z-X\sym(X^{\T}Z)$, where $\sym(X^{\T}Z)=(X^{\T}Z+Z^{\T}X)/2$. The gradient of $f_\theta$ on the manifold at $X$ is
then given by
\cite[(3.35)]{abms:2008}, \cite[Corollary 1]{chtr:2001}
\begin{align*}
\grad f_{\theta|{{\mathbb O}^{n\times k}}}(X)
&=\Pi_X\left(\frac{\partial f_{\theta}(X)}{\partial X}\right)
=\frac{\partial f_{\theta}(X)}{\partial X}-X\sym\left(X^{\T}\frac{\partial f_{\theta}(X)}{\partial X}\right) \\
&=\frac 2{[\tr(X^{\T}BX)]^{\theta}}\left[AX-\theta g_1(X)BX\right]-X\Lambda_1 \\
&\quad+\frac 1{[\tr(X^{\T}BX)]^{\theta}}\left[D-2\theta h_1(X) BX\right]-X\Lambda_2,
\end{align*}
where $\Lambda_i\in\bbR^{k\times k}$ for $i=1,2$ are symmetric and their explicit forms, although can be written out, are not  important to us. Finally, the KKT condition, also known as the first order optimality condition, is given by
$\grad f_{\theta|{{\mathbb O}^{n\times k}}}(X)=0$, or equivalently,
\begin{subequations}\label{eq:KKT}
	\begin{align} 
	&\frac 2{[\tr(X^{\T}BX)]^{\theta}}\left[AX+\frac D2-\theta f_1(X)BX\right]=X\what\Lambda,\,  \label{eq:KKT-1}\\
	&X\in\bbO^{n\times k},\,\,\what\Lambda^{\T}=\what\Lambda\in\bbR^{k\times k}. \label{eq:KKT-2}
	\end{align}
\end{subequations}
An explicit expression for $\what\Lambda$ can be obtained by pre-multiplying equation \eqref{eq:KKT-1} by
$X^{\T}$, and again it is not important to us.
Equation \eqref{eq:KKT-1} bears some similarity to NEPv \eqref{eq:NEPv},
but not quite the same because of the isolated term $D$. Next, we introduce
\begin{align}
E(X)  &=\frac 2{[\tr(X^{\T}BX)]^{\theta}}\left[A+\frac {DX^{\T}+XD^{\T}}2-\theta f_1(X)B\right]
\label{eq:E(X)}
\end{align}
and consider the following NEPv
\begin{equation}\label{eq:KKT-NEPv}
E(X)X=X\Lambda,\, X\in\bbO^{n\times k}.
\end{equation}
Pre-multiply \eqref{eq:KKT-NEPv} by $X^{\T}$ to get $\Lambda=X^{\T}E(X)X$, always symmetric.

As far as the solution $X$ to NEPv \eqref{eq:KKT-NEPv} is concerned, the scalar factor $2[\tr(X^{\T}BX)]^{-\theta}$ in
\eqref{eq:E(X)} can be merged into $\Lambda$ to give an equivalent NEPv
\begin{equation}\nonumber
\left[A+\frac {DX^{\T}+XD^{\T}}2-\theta f_1(X)B\right]X=X\Lambda,\, X\in\bbO^{n\times k}.
\end{equation}
It is interesting to notice that the power $\theta$ in $f_{\theta}$ shows up as a scalar multiplier, dictating how
much $f_1(X)B$ is subtracted.

Our first theorem establishes an equivalency relation between the KKT condition \eqref{eq:KKT} and NEPv \eqref{eq:KKT-NEPv}.

\begin{theorem}\label{thm:KKT}
	$X\in\bbO^{n\times k}$  is a KKT point, i.e., it satisfies \eqref{eq:KKT}, if and only if it is an orthonormal basis matrix of a $k$-dimensional invariant subspace of $E(X)$
	and $X^{\T}D$ is symmetric.
\end{theorem}

\begin{proof}
	Suppose that $X$ satisfies \eqref{eq:KKT}. Pre-multiply \eqref{eq:KKT-1} by
	$X^{\T}$ and then solve for $X^{\T}D$ to conclude that $X^{\T}D$ is symmetric. Next, upon using $X^{\T}X=I_k$, we have
	$$
	E(X)X=X\what\Lambda+\frac 1{[\tr(X^{\T}BX)]^{\theta}}XD^{\T}X=X\left(\what\Lambda+\frac {D^{\T}X}{[\tr(X^{\T}BX)]^{\theta}}\right)
	=: X\Lambda,
	$$
	which gives \eqref{eq:KKT-NEPv}. On the other hand, suppose that
	\eqref{eq:KKT-NEPv} holds and $X^{\T}D$ is symmetric. We expand \eqref{eq:KKT-NEPv} and rearrange the terms to get
	\begin{align*}
	\mbox{LHS of  \eqref{eq:KKT-1}}
	&=-\frac 1{[\tr(X^{\T}BX)]^{\theta}}XD^{\T}X+X\Lambda=X\left(\Lambda-\frac {D^{\T}X}{[\tr(X^{\T}BX)]^{\theta}}\right) \\
	&=: X\what\Lambda,
	\end{align*}
	which gives \eqref{eq:KKT-1} and also $\what\Lambda$ is symmetric
	because $\Lambda$ and $D^{\T}X$ are symmetric.
\end{proof}

\begin{remark}\label{rk:KKT}
	{\rm
		Theorem~\ref{thm:KKT} says $\cR(X)$ is a $k$-dimensional eigenspace of $E(X)$ if $X$ is a KKT point.
		In general when $D\ne 0$, we may not be able to make the columns of $X$ to be eigenvectors of $E(X)$ due to the requirement
		that $X^{\T}D$ has also to be symmetric. However, for the case $D=0$, if $X$ is a KKT point, then we can pick an
		orthogonal matrix $Q\in\bbR^{k\times k}$ such that the columns of $XQ$ are the eigenvectors
		of $E(X)=E(XQ)$ and at the same time $E(XQ)XQ=XQ(Q^{\T}\Lambda Q)$, implying $XQ$ is also a KKT point.
	}
\end{remark}

The following lemma gives a crucial estimate for our analysis later in this paper.

\begin{lemma}\label{lm:mono}
	For $X,\wtd X\in\bbO^{n\times k}$, let
	\begin{equation}\nonumber
	\alpha=\tr ( X^{\T}AX),\,\,
	\delta=\tr ( X^{\T}D),\,\,
	\beta=\tr ( X^{\T} B X),\,\,
	\wtd\beta=\tr (\wtd X^{\T} B\wtd X).
	\end{equation}
	If
	\begin{equation}\label{eq:tr-ineq}
	\tr(\wtd X^{\T} E(X)\wtd X)\ge\tr ( X^{\T} E(X) X),
	\end{equation}
	then
	\begin{equation}\label{eq:obj-ineq-detail}
	f_{\theta}(X)+\gamma\le g_{\theta}(\wtd X)+\frac {\tr(\wtd X^{\T} DX^{\T}\wtd X)}{[\tr(\wtd X^{\T} B\wtd X)]^{\theta}},
	\end{equation}
	where
	\begin{equation}\label{eq:gamma}
	\gamma=\frac {\alpha+\delta}{\wtd\beta^{\theta}\beta}
	\Big[(1-\theta)\beta+\theta\wtd\beta-\beta^{1-\theta}\wtd\beta^{\theta}\Big].
	\end{equation}
	Furthermore, if inequality \eqref{eq:tr-ineq} is strict, then so is  inequality \eqref{eq:obj-ineq-detail}.
\end{lemma}

As the first consequence of Lemma~\ref{lm:mono}, we have the next three theorems. These theorems lay the foundation of
our SCF iteration for NEPv \eqref{eq:KKT-NEPv} in section~\ref{sec:SCF}, which iterates the current approximation $X$ to the next
one $\wtd XQ_{\opt}$, where $Q_{\opt}$ is as specified in Theorem~\ref{thm:opt-subspace}, while
the objective value is increased.

\begin{theorem}\label{thm:mono}
	For $X,\wtd X\in\bbO^{n\times k}$, suppose either $\theta\in\{0,1\}$ or $\tr (X^{\T}AX+X^{\T}D)\ge 0$.
	If \eqref{eq:tr-ineq} holds, then
	\begin{equation}\label{eq:obj-ineq}
	f_{\theta}(X)
	\le g_{\theta}(\wtd X)+\frac {\tr(\wtd X^{\T} DX^{\T}\wtd X)}{[\tr(\wtd X^{\T} B\wtd X)]^{\theta}}
	\le g_{\theta}(\wtd X)+\frac {\|\wtd X^{\T} D\|_{\tr}}{[\tr(\wtd X^{\T} B\wtd X)]^{\theta}}.
	\end{equation}
	Furthermore, if inequality \eqref{eq:tr-ineq} is strict, then  so is the first inequality in \eqref{eq:obj-ineq}.
\end{theorem}

\begin{proof}
	In Lemma~\ref{lm:mono}, we note $\gamma\equiv 0$
	in the case $\theta\in\{0,1\}$, and $\gamma\ge 0$ in the case $\alpha+\delta=\tr (X^{\T}AX+X^{\T}D)\ge 0$.
	Hence the first inequality in \eqref{eq:obj-ineq} holds.
	To prove the second inequality in \eqref{eq:obj-ineq},
	we note, by
	von Neumann's trace inequality \cite{neum:1937} (see also \cite[p.182]{hojo:1991}, \cite[6.81]{sebe:2007}), that
	$$
	\tr(\wtd X^{\T} DX^{\T}\wtd X)
	\le \sum_{i=1}^k\sigma_i(\wtd X^{\T} D)\sigma_i(X^{\T}\wtd X)\le \sum_{i=1}^k\sigma_i(\wtd X^{\T} D)
	=\|\wtd X^{\T} D\|_{\tr},
	$$
	yielding the second inequality in \eqref{eq:obj-ineq}. 
\end{proof}

\begin{theorem}\label{thm:opt-subspace}
	Given $\wtd X\in\bbO^{n\times k}$, we have
	\begin{equation}\nonumber
	\max_{Q\in\bbO^{k\times k}} f_{\theta}(\wtd XQ)=g_{\theta}(\wtd X)+\frac {\|\wtd X^{\T} D\|_{\tr}}{[\tr(\wtd X^{\T} B\wtd X)]^{\theta}},
	\end{equation}
	and $Q_{\opt}=UV^{\T}$ is a global maximizer, where $U,\,V\in\bbO^{k\times k}$ are from
	the SVD $\wtd X^{\T}D=U\Sigma V$ \cite{govl:2013}.
\end{theorem}

\begin{proof}
	Since $\tr([\wtd XQ]^{\T}H[\wtd XQ])\equiv \tr(\wtd X^{\T} H\wtd X)$ for $H\in\{A,B\}$, we have
	$$
	\max_{Q\in\bbO^{k\times k}} f_{\theta}(\wtd XQ)=g_{\theta}(\wtd X)+\frac {\max_{Q\in\bbO^{k\times k}}\tr([\wtd XQ]^{\T}D)}{[\tr(\wtd X^{\T} B\wtd X)]^{\theta}}.
	$$
	By von Neumann's trace inequality \cite{neum:1937}, we have $\tr([\wtd XQ]^{\T}D)\le\|\wtd X^{\T} D\|_{\tr}$ for any $Q\in\bbO^{k\times k}$
	and the equality is attained for $Q=UV$ because \linebreak $\tr([\wtd XUV^{\T}]^{\T}D)=\tr(V^{\T}\Sigma V)=\|\wtd X^{\T} D\|_{\tr}$. 
\end{proof}

\begin{remark}\label{rk:mono}
	Throughout this paper, we limit $\theta\in[0,1]$. We remark here about some of the implications for $\theta$ outside
	this range. Right before we stated Theorem~\ref{thm:mono}, we emphasized the importance of Theorems~\ref{thm:mono} and
	\ref{thm:opt-subspace} to our SCF algorithm in producing the next
	approximation from the current one $X$, while
	the objective value of \eqref{eq:main-op} is increased. The key is to make $\gamma\ge 0$ in \eqref{eq:gamma}. For that,
	we require the condition of Theorem~\ref{thm:mono}: $\theta\in\{0,1\}$ or $\tr (X^{\T}AX+X^{\T}D)\ge 0$. Now if
	$\theta<0$ or $\theta>1$, then $(1-\theta)\beta+\theta\wtd\beta-\beta^{1-\theta}\wtd\beta^{\theta}<0$ unless $\beta=\wtd\beta$. Hence
	to still have $\gamma\ge 0$, we will need to assume $\tr (X^{\T}AX+X^{\T}D)=\alpha+\delta\le 0$. This observation is actually quite interesting.
	In fact, if $\tr (X^{\T}AX+X^{\T}D)\le 0$ for any $X\in\bbO^{n\times k}$, the rest of the developments in this paper are still valid
	with minor changes, namely removing the conditions imposed on initial guess in Algorithms~\ref{alg:SCF} and \ref{alg:OMA-theta}.
\end{remark}

\begin{lemma}[{\cite[Lemma 3]{zhwb:2020}}]\label{lm:maxtrace}
	Let $H\in\bbR^{k\times k}$. Then
	$
	|\tr(H)|\le\sum_{i=1}^k\sigma_i(H).
	$
	If
	$
	|\tr(H)|=\sum_{i=1}^k\sigma_i(H),
	$
	then
	$H$ is symmetric and is either positive semidefinite when $\tr(H)\ge 0$, or negative semi-definite when $\tr(H)\le 0$.
\end{lemma}

The next theorem presents necessary conditions for local or global maximizers of \eqref{eq:main-op}.

\begin{theorem}\label{thm:global-necessary}
	Let $X_{\opt}\in\bbO^{n\times k}$ be a local or global maximizer of \eqref{eq:main-op}.
	\begin{enumerate}[{\rm (a)}]
		\item If $X_{\opt}\in\bbO^{n\times k}$ is a global maximizer, then $X_{\opt}^{\T}D\succeq 0$;
		\item If $X_{\opt}^{\T}D\succeq 0$ and if $\tr(X_{\opt}^{\T}AX_{\opt}+X_{\opt}^{\T}D)\ge 0$, then $X_{\opt}$ is an orthonormal basis matrix
		of the invariant subspace associated with the $k$ largest eigenvalues of $E(X_{\opt})$.
	\end{enumerate}
\end{theorem}
  
Theorem~\ref{thm:global-necessary} presents necessary conditions for a local/global maximizer. Unfortunately,
it is not clear if they are sufficient, except in the case $D=0$ and $\theta\in\{0,1\}$ for which item~(a)
is trivially true and item~(b) is sufficient.
In fact, the case for $D=0$ and $\theta=0$ corresponds to the standard symmetric eigenvalue problem, and
the case for $D=0$ and $\theta=1$ corresponds to LDA \eqref{eq:OLDA} \cite{zhln:2010,zhln:2013}.
The next theorem sheds lights on why Theorem~\ref{thm:global-necessary} gives necessary conditions for $D\ne 0$ or
$\theta\in\{0,1\}$ but not sufficient ones.

\begin{theorem}\label{thm:global-suff?}
	If $X\in\bbO^{n\times k}$ is an orthonormal basis matrix
	of the invariant subspace associated with the $k$ largest eigenvalues of $E(X)$, then
	for any $\wtd X\in\bbO^{n\times k}$ we have
	\begin{equation}\label{eq:global-suff?}
	g_{\theta}(\wtd X)+\frac {\tr(\wtd X^{\T} DX^{\T}\wtd X)}{[\tr(\wtd X^{\T} B\wtd X)]^{\theta}}
	\le f_{\theta}(X)+\gamma,
	\end{equation}
	where $\gamma$ is defined as in \eqref{eq:gamma}. In particular, for $D=0$ and $\theta\in\{0,1\}$ we have
	$f_{\theta}(\wtd X)\le f_{\theta}(X)$, implying that $X$ is a global maximizer.
\end{theorem}

\begin{proof}
	Some minor modifications of the proof of Lemma~\ref{lm:mono} yield a proof of \eqref{eq:global-suff?}.
	If $\theta\in\{0,1\}$, then $\gamma\equiv 0$. If also $D=0$, then $f_{\theta}(\wtd X)\ge f_{\theta}(X)$.
\end{proof}

There are two reasons why Theorem~\ref{thm:global-necessary}(b) may not be sufficient, as can be seen from
inequality \eqref{eq:global-suff?}. First, $\gamma\ge 0$, and second,
$\tr(\wtd X^{\T} DX^{\T}\wtd X)\le\|\wtd X^{\T} D\|_{\tr}$. Thus
in general \eqref{eq:global-suff?} doesn't yield $f_{\theta}(\wtd X)\le f_{\theta}(X)$.

\section{The Role of $D$}\label{sec:RoleD}
When $D=0$, $f_{\theta}(XQ)\equiv f_{\theta}(X)$ for any $X\in\bbO^{n\times k}$ and $Q\in\bbO^{k\times k}$, as
in the LDA case for which $\theta=1$ as well. In such a case, $f_{\theta}$ is actually a function on the
Grassmann manifold $\scrG_k(\bbR^n)$, the collection of all $k$-dimensional subspaces in $\bbR^n$. Any
global maximizer $X_{\opt}$ of \eqref{eq:main-op} is a representative of a class
$\{X_{\opt}Q\,:\, Q\in\bbO^{k\times k}\}$ of maximizers. As a result,
maximizers are not unique. Fortunately, often any maximizer is just as good as another in practice such as
in LDA.

In general if $D\ne 0$, then $f_{\theta}(XQ)\not\equiv f_{\theta}(X)$. The global maximizers of \eqref{eq:main-op} cannot be characterized
as simple as we just did for the case $D=0$. Our goal in this section is to characterize the maximizers
of \eqref{eq:main-op} for a general $D$. In particular,  our main result imply that if $X_{\opt}$ is a global
maximizer and if $\rank(X_{\opt}^{\T}D)=k$, then $X_{\opt}$ is the unique maximizer within  $\{X_{\opt}Q\,:\, Q\in\bbO^{k\times k}\}$
in the sense that
$$
f_{\theta}(X)<f_{\theta}(X_{\opt})\,\,
\mbox{for any $X\in\{X_{\opt}Q\,:\, Q\in\bbO^{k\times k}\}$, $X\ne X_{\opt}$}.
$$

To achieve our goal, we will investigate, for a given $\cX_*\in\scrG_k(\bbR^n)$,
\begin{equation}\label{eq:op-cX*}
\max_{X\in\bbO^{n\times k},\,\cR(X)=\cX_*}\,f_{\theta}(X).
\end{equation}

\begin{lemma}\label{lm:rankOfXD}
	Given $\cX\in\scrG_k(\bbR^n)$,  the singular values of $X^{\T}D$ are independent of the choice of $X\in\bbO^{n\times k}$
	subject to $\cR(X)=\cX$, and as a result, $\rank(X^{\T}D)$ is a constant.
\end{lemma}

\begin{proof}
	Pick a particular $X_0\in\bbO^{n\times k}$ such that $\cR(X_0)=\cX$. Any $X\in\bbO^{n\times k}$
	satisfying $\cR(X)=\cX$ takes the form $X_0Q$ for some $Q\in\bbO^{k\times k}$.
	The conclusion is a simple consequence of
	$[(X_0Q)^{\T}D]^{\T}[(X_0Q)^{\T}D]=[X_0^{\T}D]^{\T}[X_0^{\T}D]$, which has nothing to do with $Q$. 
\end{proof}

Owing to this lemma, we define the rank of $\cX\in\scrG_k(\bbR^n)$ with respect to $D\in\bbR^{n\times k}$ by
$$
\rank_D(\cX_*)=\rank(X^{\T}D),
$$
where $X\in\bbO^{n\times k}$ satisfying $\cR(X)=\cX$. Our main result in this section is the next theorem.
Later we will make it more concrete in Theorem~\ref{thm:maximizer-decomp'}.

\begin{theorem}\label{thm:maximizer-decomp}
	Given $\cX_*\in\scrG_k(\bbR^n)$, the maximizer $X_{\opt}$ of \eqref{eq:op-cX*} admits the decomposition
	\begin{equation}\label{eq:maximizer-decomp}
	X_{\opt}=X_{\cX_*}+Y_{\cX_*},
	\end{equation}
	where $X_{\cX_*}$ depends on $\cX_*$ only  and $Y_{\cX_*}$ has a freedom of $\bbO^{(k-r)\times (k-r)}$, and
	$r=\rank_D(\cX_*)$.
	Moreover $\rank(X_{\cX_*})=r$ and $\rank(Y_{\cX_*})=k-r$.
\end{theorem}

Next we work to make Theorem~\ref{thm:maximizer-decomp}  concrete by constructing
$X_{\cX_*}$ and $Y_{\cX_*}$ in \eqref{eq:maximizer-decomp} explicitly. To that end, we
pick a $X_*\in\bbO^{n\times k}$ such that $\cR(X_*)=\cX_*$ and keep it fixed. Then any $X\in\bbO^{n\times k}$ satisfying $\cR(X)=\cX_*$ takes the form $X=X_*Q$ for some
$Q\in\bbO^{k\times k}$ and vice versa. With this $X_*$, \eqref{eq:op-cX*-1} can be equivalently reformulated as
\begin{equation}\label{eq:op-cX*-1}
\max_{Q\in\bbO^{k\times k}}\,f_{\theta}(X_*Q).
\end{equation}

\begin{lemma}\label{lm:opt-Q}
	Let $S\in\bbR^{k\times k}$ with SVD $S=U\Sigma V^{\T}$, where $U,\,V\in\bbO^{k\times k}$ and
	\begin{subequations}\label{eq:SVD-Sigma}
		\begin{align}
		&\Sigma=\diag(\mu_1I_{k_1},\ldots,\mu_{t-1} I_{k_{t-1}},\mu_t I_{k_t}) \quad\mbox{with}\,\, \label{eq:SVD-Sigma-1}\\
		&\mu_1>\ldots>\mu_{t-1}>\mu_t\ge 0,\,\, \sum_{i=1}^tk_i=k.  \label{eq:SVD-Sigma-2}
		\end{align}
	\end{subequations}
	Let $r=\rank(T)$, which is $k$ if $\mu_t>0$ or $k-k_t$ if $\mu_t=0$. The maximizers of
	$$
	\max_{Q\in\bbO^{k\times k}}\tr(Q^{\T}S)
	$$
	are given by
	\begin{equation}\label{eq:Qopt-form}
	Q_{\opt}=U_{(:,1:r)}V_{(:,1:r)}^{\T}+U_{(:,r+1:k)}WV_{(:,r+1:k)}^{\T},
	\end{equation}
	where $W\in\bbO^{(k-r)\times (k-r)}$ is arbitrary (by convention, if $r=k$ then $W$ is a null matrix and the second term in
	\eqref{eq:Qopt-form} disappears altogether).
\end{lemma}

\begin{remark}\label{rk:Qopt-form}
	The decomposition of $Q_{\opt}$ as the sum of two terms in \eqref{eq:Qopt-form} is constructed
	in terms of the SVD of $S$ as specified in the lemma. As they appear, both terms are SVD-dependent!
	However, they are not. In fact, the first term  $U_{(:,1:r)}V_{(:,1:r)}^{\T}$ is the subunitary factor
	of the polar decomposition of $S$ and the factor is unique \cite{li:1993b,li:1995,lisu:2002}.
	The second term represents a set of matrices of the form $U_{\bot}WV_{\bot}^{\T}$, where $W\in\bbO^{(k-r)\times (k-r)}$ is arbitrary, and
	$U_{\bot},\,V_{\bot}\in\bbO^{k\times (k-r)}$ are any orthonormal basis matrix of
	the subspaces $\cR(S)^{\bot}$, $\cR(S^{\T})^{\bot}$, respectively.
\end{remark}

With the help of Lemma~\ref{lm:opt-Q}, we present a concrete version of Theorem~\ref{thm:maximizer-decomp} in
Theorem~\ref{thm:maximizer-decomp'} below.

\begin{theorem}\label{thm:maximizer-decomp'}
	Given $X_*\in\bbO^{n\times k}$, let $X_*^{\T}D=U\Sigma V^{\T}$ be the SVD of $X_*^{\T}D$, where $U,\,V\in\bbO^{k\times k}$ and
	$\Sigma$ as in \eqref{eq:SVD-Sigma}.
	Let $r=\rank(X_*^{\T}D)$.
	For any maximizer $Q_{\opt}$ of \eqref{eq:op-cX*-1},
	\begin{equation}\label{eq:Xopt-form}
	X_*Q_{\opt}=X_*U_{(:,1:r)}V_{(:,1:r)}^{\T}+X_*U_{(:,r+1:k)}WV_{(:,r+1:k)}^{\T},
	\end{equation}
	for which the first term $X_*U_{(:,1:r)}V_{(:,1:r)}^{\T}$ has rank $r$
	and the second term has rank $k-r$ and a freedom in $W\in\bbO^{(k-r)\times (k-r)}$. Moreover,
	\begin{equation}\nonumber
	[X_*Q_{\opt}]^{\T}D=V\Sigma V^{\T}\succ 0,\,\,
	\tr([X_*Q_{\opt}]^{\T}D)=\|X_*^{\T}D\|_{\tr}.
	\end{equation}
\end{theorem}

\begin{proof}
	It is not hard to see that
	$$
	\max_{Q\in\bbO^{k\times k}}\,f_{\theta}(X_*Q)=g_{\theta}(X_*)+\frac {\tr([X_*Q]^{\T}D)}{[\tr(X_*^{\T}BX_*)]^{\theta}}.
	$$
	Hence the optimizers of \eqref{eq:op-cX*-1} are the same as those of
	\begin{equation}\nonumber
	\max_{Q\in\bbO^{k\times k}}\tr([X_*Q]^{\T}D)=\max_{Q\in\bbO^{k\times k}}\tr(Q^{\T}X_*^{\T}D).
	\end{equation}
	Hence $Q_{\opt}$ takes the form of \eqref{eq:Qopt-form}, yielding \eqref{eq:Xopt-form} by Lemma~\ref{lm:opt-Q} with $S=X_*^{\T}D$.
	The rest of claims of the theorem
	are simple consequences. 
\end{proof}

Now we are ready to prove Theorem~\ref{thm:maximizer-decomp}.

\begin{proof}[Proof of Theorem~\ref{thm:maximizer-decomp}]
	For any particularly chosen $X_*\in\bbO^{n\times k}$ satisfying $\cR(X_*)=X_*$, adopting the notation of Theorem~\ref{thm:maximizer-decomp}, we find
	$X_{\opt}=X_*Q_{\opt}$ as given by \eqref{eq:Xopt-form}. We claim that the first term $X_*U_{(:,1:r)}V_{(:,1:r)}^{\T}$
	there depends on $\cX_*$ only, i.e., it stays the same for any different choice of $X_*$, although it is constructed by
	a particularly chosen $X_*$. First we note $r=\rank(X_*^{\T}D)$
	depends on $\cX_*$ only by Lemma~\ref{lm:rankOfXD}. Second, the product $U_{(:,1:r)}V_{(:,1:r)}^{\T}$ does not change with the inherent
	variations in SVD, as we argued in Remark~\ref{rk:Qopt-form}.
	Third, suppose a different $\wtd X_*\in\bbO^{n\times k}$ satisfying $\cR(\wtd X_*)=X_*$ is chosen. Then $\wtd X_*=X_*\wtd Q$ for some
	$\wtd Q\in\bbO^{k\times k}$. We have
	$$
	\wtd X_*^{\T}D=(X_*\wtd Q)^{\T}D=\wtd Q^{\T}X_*^{\T}D=(\wtd Q^{\T}U)\Sigma V^{\T}.
	$$
	As we just argued that the product ``$U_{(:,1:r)}V_{(:,1:r)}^{\T}$ does not change with the inherent
	variations in SVD'', we conclude that the  first term in \eqref{eq:Xopt-form} corresponding to $\wtd X_*$ is given by
	$$
	\wtd X_*(\wtd Q^{\T}U)_{(:,1:r)}V_{(:,1:r)}^{\T}=(X_*\wtd Q)({\wtd Q}^{\T}U_{(:,1:r)})V_{(:,1:r)}^{\T}=X_*U_{(:,1:r)}V_{(:,1:r)}^{\T},
	$$
	having nothing to do with $\wtd Q$, as expected. 
\end{proof}

The second terms in \eqref{eq:maximizer-decomp} and its concrete version in \eqref{eq:Xopt-form} disappear
altogether if $r:=\rank_D(\cX_*)=k$. Hence we have the following corollary.

\begin{corollary}\label{cor:Qopt-form}
	Problem \eqref{eq:op-cX*} has a unique maximizer if $\rank_D(\cX_*)=k$.
\end{corollary}

\section{Self-Consistent Field Iteration}\label{sec:SCF}
In what follows we will limit problem \eqref{eq:main-op} to the case that
\begin{equation}\label{eq:positive}
\framebox{
	\parbox{10cm}{
		there exists $X\in\bbO^{n\times k}$ such that $\tr(X^{\T}AX+X^{\T}D)\ge 0$.
	}
}
\end{equation}
This assumption ensures that, at optimality, the objective value is nonnegative. It evidently holds if $A\succeq 0$,
because $\tr(X^{\T}AX+X^{\T}D)\ge 0$ for $X=UV^{\T}$ where $U,\, V$ are from the SVD of $D=U\Sigma V^{\T}$.

Assumption \eqref{eq:positive} is not really needed  for our main results
to hold in the case when $\theta\in\{0,1\}$, however. But we argue that having \eqref{eq:positive} doesn't lose any generality, either, even
for $\theta\in\{0,1\}$. In fact, it can be verified that for $X\in\bbO^{n\times k}$
\begin{align*}
f_0(X)&=\tr(X^{\T}[A+\alpha I_n]X+X^{\T}D)-k\alpha, \\
f_1(X)&=\frac {\tr(X^{\T}[A+\alpha B]X+X^{\T}D)}{\tr(X^{\T}BX)}-\alpha.
\end{align*}
By choosing a sufficiently large $\alpha>0$, we can make $\tr(X^{\T}[A+\alpha I_n]X+X^{\T}D)\ge 0$
and $\tr(X^{\T}[A+\alpha B]X+X^{\T}D)\ge 0$, equivalently reducing problem \eqref{eq:main-op} for $\theta\in\{0,1\}$ to the case of
it satisfying assumption \eqref{eq:positive}.

\subsection{SCF}
Based on the KKT condition in Theorem~\ref{thm:KKT}, Theorems~\ref{thm:mono} and \ref{thm:opt-subspace}, and the necessary conditions in Theorem~\ref{thm:global-necessary} for a local/global maximizer,
an SCF iteration as outlined in Algorithm~\ref{alg:SCF} is rather natural.

\begin{algorithm}[H]
	\caption{SCF iteration for problem \eqref{eq:main-op} satisfying \eqref{eq:positive}} \label{alg:SCF}
	\begin{algorithmic}[1]
		\REQUIRE $X_0 \in \bbO^{n\times k}$ such that $\tr(X_0^{\T}AX_0+X_0^{\T}D)\ge 0$ if $0<\theta<1$, otherwise not required;
		\ENSURE  a maximizer of \eqref{eq:main-op}.
		\FOR{$i=1,2,\ldots$ until convergence}
		\STATE construct $E_i = E(X_{i-1})$ as in \eqref{eq:E(X)};
		\STATE compute the partial eigen-decomposition
		$E_i \what X_i = \what X_i \Lambda_{i-1}$ for the $k$ largest eigenvalues of $E_i$ and their associated eigenvectors, or
		simply some $\what X_i\in\bbO^{n\times k}$ such that
		$\tr(\what X_i^{\T}E_i\what X_i)>\tr(X_{i-1}^{\T}E_iX_{i-1})$;
		\STATE compute SVD: $\what X_i^{\T}D=U_i \Sigma_i V_i^{\T}$;
		\STATE $X_i=\what X_i U_iV_i^{\T}$;
		\ENDFOR
		\RETURN the last $X_i$ as a maximizer of \eqref{eq:main-op}.
	\end{algorithmic}
\end{algorithm}

A few comments are in order for Algorithm~\ref{alg:SCF}.
\begin{enumerate}[1)]
	\item It is required initially $\tr(X_0^{\T}AX_0+X_0^{\T}D)\ge 0$ if $0<\theta<1$ but otherwise not necessary for $\theta\in\{0,1\}$,
	in order to ensure that $\{f_{\theta}(X_i)\}$ is monotonically increasing (see Theorem~\ref{thm:convg1} in the next subsection). An immediate
	question arises as to what to do if we don't have such an initial $X_0$ in the case $0<\theta<1$. Our suggestion is to
	set $\theta=0$ or $1$ and iterate until some $X_i$ with $\tr(X_i^{\T}AX_i+X_i^{\T}D)\ge 0$ and then switch back to the original
	$\theta$.
	
	\item At line 3, we offer two options to obtain $\what X_i$. Evidently, $\what X_i$ associated with the $k$ largest eigenvalues of
	$E_i$ maximizes $\tr(\what X_i^{\T}E_i\what X_i)$. But as we will show later in Theorem~\ref{thm:convg1} that
	the objective  value will still increase as long as $\tr(\what X_i^{\T}E_i\what X_i)>\tr(X_{i-1}^{\T}E_iX_{i-1})$.
	This is an important observation, especially for large scale cases where an iterative method
	has to be used to compute the partial eigen-decomposition of $E_i$ and the convergence to a
	very accurate partial eigen-decomposition is rather slow. When that is the case, we can afford to
	compute  partial eigen-decompositions with gradually increased accuracy as
	the for-loop progresses, namely, use less accurate partial eigen-decompositions at the beginning many for-loops
	to save work and more and more accurate partial eigen-decompositions as $X_i$
	comes closer and closer to the target. Such an adaptive strategy is a delicate issue and often the best strategy
	is problem-dependent. Further study on this is out of the scope of this paper and should be pursued elsewhere.
	\item Lines 4 and 5 execute $\max_Qf(\what X_iQ)$ yielding $X_i$ according to Theorem~\ref{thm:maximizer-decomp'} (with $W=I$ in \eqref{eq:Qopt-form} always if needed).
	$X_i$ is not uniquely defined if $\rank(\what X_i^{\T}D)<k$. But that non-uniqueness doesn't affect
	the corresponding objective value.
	\item A natural stopping criterion to end the for-loop is to use the normalized residual of NEPv \eqref{eq:KKT-NEPv}:
	\begin{equation}\label{eq:stop}
	\frac {[\tr(X_i^{\T}BX_i)]^{\theta}}{2\sqrt k}\cdot\frac {\|E(X_i)X_i-X_i(X_i^{\T}E(X_i)X_i)\|_{\F}}{\|A\|_2+\theta\,|f_1(X_i)|\,\|B\|_2+\|D\|_2}\le \tol,
	\end{equation}
	where $\tol$ is a preset tolerance. For computational convenience, it will be just fine to replace
	the spectral norms $\|A\|_2$, $\|B\|_2$, and $\|D\|_2$ by their corresponding $1$-norm.
\end{enumerate}

Theorem~\ref{thm:property-immediate} presents two properties about each approximation $X_i$.
\begin{theorem}\label{thm:property-immediate}
	Let $\{X_i\}_{i=0}^{\infty}$ be generated by Algorithm~\ref{alg:SCF}.
	\begin{enumerate}[{\rm (a)}]
		\item $X_i^{\T}D$ is symmetric and positive semi-definite, and $\tr(X_i^{\T}D)=\|X_i^{\T}D\|_{\tr}$.
		\item If the eigenvalue gap
		$$
		\lambda_{k}(E(X_{i-1}))-\lambda_{k+1}(E(X_{i-1}))>0,
		$$
		then any two orthonormal eigenbasis matrices $\what X_{i}$ and $\what Y_{i}$ associated
		with  $k$ largest eigenvalues of $E(X_{i-1})$ satisfy $\what Y_{i}=\what X_{i}Q$
		for some  $Q\in\mathbb{O}^{k\times k}$.
		Furthermore, if, additionally, $\rank(D^{\T}{\what X_{i}})=k$ (which is  independent of $Q$), then the
		next approximation $X_{i}$ from line 4 of Algorithm~\ref{alg:SCF} is uniquely determined regardless of
		any inherent freedom in the SVD.
	\end{enumerate}
\end{theorem}

\begin{proof}
	The conclusions in item~(a) follows from $X_i^{\T}D=V_i\Sigma_i V_i^{\T}$.
	
	Consider item~(b). Since the eigenvalue gap is positive, the eigenspace associated with the $k$ largest eigenvalues of
	$E(X_{i-1})$ is unique  \cite[p. 244]{stew:2001a}, and thus the first claim $\what Y_{i}=\what X_{i}Q$ follows.
	The second claim is a consequence of Theorem~\ref{thm:maximizer-decomp'}.
\end{proof}

\subsection{Convergence Analysis}

Much of our analysis is similar to \cite{zhwb:2020} which is for the case $A=0$ and $\theta=1/2$.
Notably, the complete characterization on what the limits of $X_i$ may look like in the rank-deficient situation
$\rank(X_*^{\T}D)$ in Theorem~\ref{thm:convg2} is new even for the special case.

We start by presenting a few  basic convergence properties of our SCF iteration (Algorithm~\ref{alg:SCF})
for solving \eqref{eq:main-op}.

\begin{theorem}\label{thm:convg1}
	Let  the sequence $\{X_i\}_{i=0}^{\infty}$  be generated by the SCF iteration (Algorithm~\ref{alg:SCF}).
	The following statements hold.
	\begin{enumerate}[{\rm (a)}]
		\item The sequence $\{f_{\theta}(X_i)\}$ is monotonically increasing and convergent;
		\item  If		
		\begin{equation}\label{eq:inexacteig}
		\tr(\what X_i^{\T}E(X_{i-1})\what X_i) < \tr(X_{i-1}^{\T}E(X_{i-1}) X_{i-1}),
		\end{equation}
		then $f_{\theta}(X_{i-1})<f_{\theta}(X_i)$;
		\item $\{X_i\}_{i=0}^{\infty}$ has a convergent subsequence $\{X_i\}_{i\in \cI}$, where $\cI$ is subset of $\{0,1,2,\ldots\}$, and
		\begin{equation}\nonumber
		\lim_{i\to\infty,~i\in\cI}f_{\theta}(X_i)=\lim_{i\to\infty}f_{\theta}(X_i),
		\end{equation}
		i.e., along any convergent subsequence, the limit of the objective value is the same.
		\item Let  $\{X_i\}_{i\in\cI}$  be any convergent subsequence of $\{X_i\}_{i=0}^{\infty}$ with an accumulation
		point $X_*$ satisfying
		\begin{equation}\label{eq:gapcondition}
		\zeta=\lambda_k(E(X_*))-\lambda_{k+1}(E(X_*))>0.
		\end{equation}
		Then $X_*$ satisfies the first order optimality condition in
		\eqref{eq:KKT} and  the necessary condition  in Theorem~\ref{thm:global-necessary} for a local/global maximizer.
	\end{enumerate}
\end{theorem}

To further analyze the convergence  of the sequence $\{X_i\}_{i=0}^{\infty}$,
we now introduce the distance measure between two subspaces $\cX$ and $\cY$ of dimension $k$ \cite[p.95]{sun:1987}.
Let
$\cX=\cR(X)$ and $\cY=\cR(Y)$, where $X,\,Y\in\bbR^{n\times k}$ with
$X^{\T}X=Y^{\T}Y=I_k$. The canonical angles
$\theta_1(\cX,\cY)\ge\cdots\ge\theta_1(\cX,\cY)$ between $\cX$ and $\cY$ are defined by
$$
0\le\theta_i(\cX,\cY):=\arccos \sigma_i(X^{\T}Y)\le\frac {\pi}2 \quad\mbox{for $1\le i\le k$},
$$
and accordingly, 
$
\Theta(\cX,\cY)=\diag(\theta_1(\cX,\cY),\dots,\theta_k(\cX,\cY)).
$
It is known that
\begin{equation}\label{eq:sinTheta}
\dist_2(\cX,\cY):=\|\sin\Theta(\cX,\cY)\|_2
\end{equation}
is a unitarily invariant metric \cite[p.95]{sun:1987} on
Grassmann manifold
$\scrG_k(\bbR^n)$, the collection of all $k$-dimensional subspaces in $\bbR^n$.

The following lemma is an equivalent restatement of \cite[Lemma 4.10]{moso:1983}
(see also \cite[Proposition 7]{kaqi:1999}) in the context of Grassmann manifold $\scrG_k(\bbR^n)$
(see also the supplementary material of \cite{zhwb:2020}).

\begin{lemma}[{\cite[Lemma 4.10]{moso:1983}}]\label{lem:isolatedconvg}
	Let $\cX_*\in \scrG_k(\bbR^n)$ be an isolated accumulation point
	of the sequence $\{\cX_i\in\scrG_k(\bbR^n)\}_{i=0}^{\infty}$, in  the metric \eqref{eq:sinTheta}, such that,
	for every subsequence $\{\cX_i\}_{i\in\cI}$ converging to $\cX_*$, there is
	an infinite subset $\widehat{\cI}\subseteq \cI$ satisfying
	$\{ \dist_2(\cX_i,\cX_{i+1})\}_{i\in\what\cI}\rightarrow 0$.
	Then the entire sequence $\{\cX_i\}_{i=0}^{\infty}$ converges to $\cX_*$.
\end{lemma}

\begin{theorem}\label{thm:convg2}
	Let  the sequence $\{X_i\}_{i=0}^{\infty}$  be generated by the SCF iteration (Algorithm~\ref{alg:SCF}),
	and let $X_*$ be an accumulation point of $\{X_i\}_{i=0}^{\infty}$.
	Suppose that
	$\cR(X_*)$  is an isolated accumulation point
	in the metric \eqref{eq:sinTheta} of $\{\cR(X_i)\}_{i=0}^{\infty}$, and that the eigenvalue gap
	assumption \eqref{eq:gapcondition} holds.
	Let $r:=\rank(X_*^{\T}D)$.
	\begin{enumerate}[{\rm (a)}]
		\item The entire sequence $\{\cR(X_i)\}_{i=0}^{\infty}$ converges to $\cR(X_*)$.
		\item If also  $r=k$, then $\{X_i\}_{i=0}^{\infty}$ converges to $X_*$ (in the standard Euclidean metric).
		\item In general if $r<k$, $\{X_i\}_{i=0}^{\infty}$ converges to the set,
		in the notation of Theorem~\ref{thm:maximizer-decomp'} ($U=V$ because $X_*^{\T}D\succeq 0$),
		\begin{equation}\label{eq:bbX*}
		\bbX_*=\left\{X_*V_{(:,1:r)}V_{(:,1:r)}^{\T}+X_*V_{(:,r+1:k)}WV_{(:,r+1:k)}^{\T}\,:\,W\in\bbO^{(k-r)\times (k-r)}\right\}
		\end{equation}
		in the sense that
		\begin{equation}\label{eq:Xi2what}
		\min_{X\in\bbX_*} \|X_i-X\|_2\to 0\,\,\,\mbox{as $i\to\infty$}.
		\end{equation}
	\end{enumerate}
\end{theorem}

What is remarkable about Theorem~\ref{thm:convg2} is that we start with an accumulation point $X_*$ which always exists because
$\bbO^{n\times k}$ is a bounded set in $\bbR^{n\times k}$ and thus is compact, and end up with the conclusions that
$\{\cR(X_i)\}_{i=0}^{\infty}$
converges to $\cR(X_*)$ and that $\{X_i\}_{i=0}^{\infty}$ converges to $X_*$ under the conditions that
$\rank(X_*^{\T}D)=k$ and $\lambda_k(E(X_*))>\lambda_{k+1}(E(X_*))$. In general, $X_i$ falls into $\bbX_*$
of \eqref{eq:bbX*} as $i\to\infty$.

The set $\bbX_*$ is uniquely determined by $\cR(X_*)$, not dependent of the particular accumulation point $X_*$.

A quantitative convergence estimate like \cite[ineq. (24)]{zhwb:2020} can be derived, too. But the estimated constant is
inevitably too big to be of any use. For that reason, we will simply skip it.

\section{Application to Multi-view Learning}\label{sec:MvSL}

Different from classical learning, multi-view learning aims to learn from multiple views  of the same object
in order to leverage their complementary
and redundant information to boost learning performances \cite{baltruvsaitis2018multimodal}. For example,
in the classification of Internet advertisement on Internet pages \cite{kushmerick1999learning},
the geometry of the image (if available) as well as phrases occuring in the URL, the image's URL and alt text,
the anchor text, and words occurring near the anchor text are considered as different views of a page.
Due to the heterogeneity of multiple views, learning from multi-view data is challenging,
even though they conceal more information. To narrow  the  heterogeneity gap \cite{peng2019cm},
multi-view subspace learning is the most popularly studied methodology by learning proper representations of the multiple views
in a common subspace.
In what follows, we will first briefly introduce the problem formulation of multi-view subspace learning and related works,
and then propose our new model and an efficient alternating iterative method based on our earlier SCF iteration Algorithm~\ref{alg:SCF} to the model.

\subsection{Problem Formulation and Related Work} \label{sec:multiview-related}
Multi-view subspace learning seeks a common latent space via some unknown transformation on each view so that certain
learning criteria over the given multi-view dataset is optimized with respect to these transformations.
Let $\{ (\bz_i^{(1)},\ldots,\bz_i^{(v)}, \by_i) \}_{i=1}^m$ be a multi-view dataset of $v$ views and $m$ instances,
where the $i$th data points $\bz_i^{(s)} \in \bbR^{n_s}$ of all views ($1\le s\le v$) share the same class label
$\by_i \in \{0,1\}^{c}$ of $c$ classes, where
the $r$th entry $(\by_i)_{(r)}=1$ if the $i$th data points belong to class $r$ and otherwise $(\by_i)_{(r)} = 0$.
Linear transformations are often used to perform feature extraction.  Specifically, we look for
projection matrix $P_s \in \bbR^{n_s \times k}$ for view $s$ to transform $\bz_i^{(s)}$ from $\bbR^{n_s}$
to $\bu_i^{(s)} = P_s^{\T} \bz_i^{(s)}$ in the common space $\bbR^{k}$. Represent the $m$ data points of view  $s$ by
$Z_s = [\bz_1^{(s)},\ldots,\bz_m^{(s)}] \in \bbR^{n_s \times m }$ and its latent representation by
$U_s = [\bu_1^{(s)} ,\ldots,\bu_m^{(s)} ] = P_s^{\T} Z_s \in \bbR^{k \times m}$.

Several existing methods \cite{cao2017generalized,sun2015multiview,sharma2012generalized} have explored
both the inter-view correlations and the intra-view class separability from the labeled multi-view data.
Some important statistical quantities are summarized as follows:
\begin{enumerate}
	\item the sample cross-covariance matrices $C_{s,t} = \frac{1}{m} Z_s H_m Z_t^{\T}$, and, in particular,
	the covariance matrices $C_{s,s} = \frac{1}{m} Z_s H_m Z_s^{\T}$,
	\item the between-class scatter matrix $S_{\rm b}^{(s)}=Z_s (Y^{\T} \Sigma^{-1} Y - \frac{1}{m} \bone_m \bone_m^{\T})$,
	\item the within-class scatter matrix $S_{\rm w}^{(s)} = Z_s (I_m - Y^{\T} \Sigma^{-1} Y) Z_s^{\T}$,
	\item the class centers scatter matrix across views $ M_{s,t} = Z_s Y^{\T} \Sigma^{-1} H_c \Sigma^{-1} Y Z_t^{\T}$,
\end{enumerate}
where centering matrix $H_m = I_m - \frac{1}{m} \bone_m \bone_m^{\T}$, label matrix $Y = [\by_1,\ldots, \by_m]$, and $\Sigma = Y Y^{\T}$.

Most existing methods are often formulated as a trace maximization problem
\begin{align}
\max_{P^{\T} B P = I_k} \tr(P^{\T} A P) \label{op:gev}
\end{align}
which is equivalent to a generalized eigenvalue problem (GEP) for matrix pencil $A-\lambda B$ \cite{demm:1997,govl:2013},
where $A=A^{\T}$ and $B=B^{\T}\succ 0$ have the following block structures
\begin{equation}\label{eq:AB4MvSL}
A=\kbordermatrix{ &\sss n_1 &\sss n_2 &\sss \cdots &\sss n_v\\
	\sss n_1 & A_{1,1} & A_{1,2} & \hdots & A_{1,n_v}\\
	\sss n_2 & A_{2,1} & A_{2,2} & \hdots & A_{2,n_v}\\
	\sss \vdots & \vdots & \vdots & & \vdots \\
	\sss n_v & A_{n_v,1} & A_{n_v,2} & \hdots & A_{n_v,n_v}        }, \quad
B=\kbordermatrix{ &\sss n_1 &\sss n_2&\sss \cdots &\sss n_v\\
	\sss n_1 & B_1 &  \\
	\sss n_2 &  & B_2 & \\
	\sss \vdots &  &  & \ddots &  \\
	\sss n_v &  &  &  & B_{n_v} },
\end{equation}
with $A_{s,t}$ and $B_s$ judiciously taken to be $C_{s,t}$, $M_{s,t}$, $S_{\rm b}^{(s)}$, and $S_{\rm w}^{(s)}$ as we will detail in a moment, and
accordingly,
\begin{equation}\label{eq:P-parts}
P=[P_1^{\T},P_2^{\T},\ldots P_{n_v}^{\T}]^{\T},\,\, P_s\in\bbR^{n_s\times k}\,\,\mbox{for $1\le s\le v$},
\end{equation}
and $1\le k\le\min _s n_s$.
For example,
\begin{itemize}
	\item multiset canonical correlation analysis (MCCA) \cite{via2007learning} is (\ref{op:gev})
	with $A_{s,t} = C_{s,t}$ and $B_{s} = C_{s,s}, \forall s, t$,
	\item generalized multi-view analysis (GMA) \cite{sharma2012generalized} is (\ref{op:gev})
	with $A_{s,t} = \alpha C_{s,t}, \forall s\not=t$, $A_{s,s} = S_{\rm b}^{(s)}$, and $B_s=S_{\rm w}^{(s)}, \forall s$,
	\item multi-view linear discriminant analysis (MLDA) \cite{sun2015multiview} is (\ref{op:gev})
	with $A_{s,t} = \alpha C_{s,t}, \forall s\not=t$, $A_{s,s} = S_{\rm b}^{(s)}$, and $B_s=C_{s,s}, \forall s$, and
	\item multi-view modular discriminant analysis (MvMDA) \cite{cao2017generalized} is (\ref{op:gev})
	with $A_{s,t} = M_{s,t}$ and $B_s=S_{\rm w}^{(s)}, \forall s, t$,
\end{itemize}
where $\alpha \geq 0$ is a pre-defined parameter to weigh the importance of class separability. As defined in these methods,
only $B\succeq 0$ is guaranteed, and there is a possibility that $B$ may be singular. When that happens, often $B$ is regularized by
adding $\gamma I$ with a tiny $\gamma$ to it. In MCCA, $A\succeq 0$ always, but it may be indefinite for the other three.
In all methods, the diagonal blocks of
$A$ are always positive semi-definite.

\subsection{Proposed Model and Alternating Iteration}
We propose a new formulation for supervised multi-view subspace learning as
\begin{equation}\label{eq:OMA}
\max_{P_s^{\T}P_s=I_k,\, \forall s=1,2,\ldots,v}\,
f_{\theta}(P) \quad\mbox{with}\quad f_{\theta}(P):=\frac {\tr(P^{\T}AP)}{[\tr(P^{\T}BP)]^{\theta}},
\end{equation}
where $A$ and $B$ generally will have the block structures as in \eqref{eq:AB4MvSL} with each block taken to be $C_{s,t}$, $M_{s,t}$,
$S_{\rm b}^{(s)}$, or $S_{\rm w}^{(s)}$, depending on learning scenarios as the above existing methods, and
$0\le\theta\le 1$. Function $f_{\theta}(P)$ is well-defined if at least one of the inequalities $\rank(B_s)>n_s-k$ for $1\le s\le v$
is valid.

Comparing with (\ref{op:gev}), the new model (\ref{eq:OMA}) possesses two unique properties:
\begin{enumerate}
	\item Linear projection matrices $P_{s}$ are orthonormal, $ \forall s=1,\ldots,v$.
	This is a preferred property for metric preservation and data visualization, and has been explored
	for unsupervised learning in, e.g., \cite{cugh:2015,zhwb:2020}. However, orthogonality constraints are incompatible
	with constraint in (\ref{op:gev}) if both are imposed;
	\item Trace ratio formulation (\ref{eq:OMA}) is a more essential formulation for
	general feature extraction problem than ratio trace formulation (\ref{op:gev}) \cite{wang2007trace}
	since it naturally solves the above-mentioned incompatible issue. The introduced $\theta$
	can further adjust the relative importance of $\tr(P^{\T} A P)$ against that of $\tr(P^{\T} B P)$. In our later numerical
	experiments, we will investigate the impact of $\theta$ in terms of classification accuracy.
\end{enumerate}

Model \eqref{eq:OMA} is a maximization problem over the product of $v$ Stiefel manifolds $\bbO^{n_s\times k}$.
The KKT conditions can be derived straightforwardly by examining the partial gradients with respect to $P_s$ on
$\bbO^{n_s\times k}$ along the line of derivations in section~\ref{sec:KKT}. In fact, for any fixed $s$, we rewrite
the objective of \eqref{eq:OMA} as
\begin{subequations}\label{eq:wrtPs}
	\begin{equation}\label{eq:wrtPs-f}
	f_{\theta}(\{P_s\})=\frac {\tr(P_{s}^{\T}\what A_{s} P_{s})+\tr(P_{s}^{\T}\what D_{s})}
	{[\tr(P_{s}^{\T}\what B_{s} P_{s})]^{\theta}},
	\end{equation}
	where
	\begin{gather}
	\what A_{s}=A_{ss}+(\alpha_{s}/k) I_{n_{s}}, \quad \what B_{s}=B_{s}+(\beta_{s}/k) I_{n_{s}},\quad\mbox{with} \label{eq:wrtPs-1}\\
	\alpha_{s}=\tr(P_{[s]}^{\T}A_{[s]}P_{[s]}), \,\,
	\beta_{s}=\sum_{s'\ne s}\tr(P_{s'}^{\T}B_{s'}P_{s'}), \,\,
	\what D_{s}=2\sum_{s'\ne s} A_{s, s'}P_{s'}, \label{eq:wrtPs-2}
	\end{gather}
\end{subequations}
$A_{[s]}$ is $A$ after crossing out its $s$ block-row and block-column, and $P_{[s]}$ is $P$ of \eqref{eq:P-parts}
after crossing out its $s$-block.
The dependency of $\what A_{s}$, $\what B_{s}$, and $\what D_{s}$ on $P_{s'}$ for $s'\ne s$ is suppressed for clarity.
The KKT conditions of \eqref{eq:OMA} can be made to consist of $v$ coupled NEPv like \eqref{eq:KKT-NEPv}: for $1\le s\le v$
\begin{subequations}\label{eq:OMA-KKT}
	\begin{equation}\label{eq:Es(Ps)-NEPv}
	E_s(P_s)\,P_s=P_s\Lambda_s,\,\,P_s\in\bbO^{n_s\times k},
	\end{equation}
	where
	\begin{equation}\label{eq:Es(Ps)}
	E_s(P_s)=\frac 2{[\tr(P_s^{\T}\what B_{s}P_s)]^{\theta}}
	\left[\what A_{s}+\frac {\what D_{s}P_s^{\T}+P_s\what D_{s}^{\T}}2
	-\theta f_1(\{P_{s'}\})\what B_{s}\right].
	\end{equation}
\end{subequations}
They are coupled because of the dependency of $\what A_{s}$, $\what B_{s}$, and $\what D_{s}$ on $P_{s'}$ for $s'\ne s$. Individually,
\eqref{eq:OMA-KKT} is the KKT condition of
\begin{equation}\label{eq:OMA-sub-II''}
\max_{P_{s}^{\T}P_{s}=I_k}\frac {\tr(P_{s}^{\T}\what A_{s} P_{s})+\tr(P_{s}^{\T}\what D_{s})}
{[\tr(P_{s}^{\T}\what B_{s} P_{s})]^{\theta}}, \quad
\mbox{given $P_{s'}$ for $s'\ne s$}.
\end{equation}

Along the line of reasoning in section~\ref{sec:KKT}, we can get the next theorem, as an extension of Theorem~\ref{thm:global-necessary}.

\begin{theorem}\label{thm:OMA-global-necessary}
	Let $\{P_s^{\opt}\in\bbO^{n_s\times k}\}_{s=1}^v$ be a local or global maximizer of \eqref{eq:OMA} and let
	$\what A_{s}$, $\what B_{s}$, and $\what D_{s}$ in \eqref{eq:OMA-KKT} be evaluated at $\{P_s^{\opt}\}_{s=1}^v$.
	\begin{enumerate}[{\rm (a)}]
		\item If $\{P_s^{\opt}\}_{s=1}^v$ is a global maximizer, then $\big(P_s^{\opt}\big)^{\T}\what D_s\succeq 0$ for $1\le s\le v$;
		\item Suppose $f_{\theta}(\{P_s^{\opt}\}_{s=1}^v)\ge 0$. If $\big(P_s^{\opt}\big)^{\T}\what D_s\succeq 0$,
		then $P_s^{\opt}$ is an orthonormal basis matrix
		of the invariant subspace associated with the $k$ largest eigenvalues of $E_s(P_s^{\opt})$.
	\end{enumerate}
\end{theorem}

\begin{proof}
	If $\{P_s^{\opt}\}_{s=1}^v$ is a global maximizer, then $P_s^{\opt}$ for a fixed $s$ is a global maximizer
	of \eqref{eq:OMA-sub-II''} with $\what A_{s}$, $\what B_{s}$, and $\what D_{s}$ evaluated at
	$\{P_{s'},\,s'\ne s\}$. Item (a) is a consequence of Theorem~\ref{thm:global-necessary}(a).
	
	Again with $\what A_{s}$, $\what B_{s}$, and $\what D_{s}$ evaluated at
	$\{P_{s'},\,s'\ne s\}$, $f_{\theta}(\{P_s^{\opt}\}_{s=1}^v)\ge 0$ implies
	$\tr([P_s^{\opt}]^{\T}\what A_{s} P_s^{\opt})+\tr([P_s^{\opt}]^{\T}\what D_{s})\ge 0$.
	Apply Theorem~\ref{thm:global-necessary}(b) to \eqref{eq:OMA-sub-II''} to conclude the proof of item~(b).
\end{proof}

\subsection{Alternating Iteration}

Although
generic optimization methods for optimizing a
smooth function over the product of the Stiefel manifolds $\bbO^{n_s\times k}$ are available; for example, classical optimization algorithms such as
the steepest descent gradient or the trust-region methods over the Euclidean space have been extended to the general
Riemannian manifolds in, for example, \cite{abms:2008},
they do not make use of the special trace-fractional structure. In what follows, we propose to
solve \eqref{eq:OMA} by maximizing its objective alternatingly over $\{P_s\in\bbO^{n_s\times k}\}_{s=1}^v$ in either
the Jacobi-style or Gauss-Seidel-style updating scheme as outlined in Algorithm~\ref{alg:OMA-theta},
where the SCF iteration in Algorithm~\ref{alg:SCF} serves as
the computational engine to solve each subproblem \eqref{eq:OMA-sub-II''} over just one $P_s$.
Problem~\eqref{eq:OMA-sub-II''} is in the form of \eqref{eq:main-op} for which  Algorithm~\ref{alg:SCF}
is provably convergent, provided it starts with an initial approximation at which the numerator
$\tr(P_{s}^{\T}\what A_{s} P_{s})+\tr(P_{s}^{\T}\what D_{s})\ge 0$.

\begin{algorithm}[t]
	\caption{OMvSL$\theta$: Orthogonal Multi-view Subspace Learning via $\theta$-Trace Ratio }\label{alg:OMA-theta}
	\begin{algorithmic}[1]
		\REQUIRE $A$ and $B$ in \eqref{eq:AB4MvSL}, $1\le k\le\min_s n_s$,
		and a tolerance $\epsilon$;
		
		\ENSURE  $\{P_s\in\bbO^{n_s\times k}\}$ that approximately solves (\ref{eq:OMA}).
		
		\STATE pick $\{P_s^{(0)}\in\bbO^{n_s\times k}\}_{s=1}^v$ satisfying $\tr([P^{(0)}]^{\T}AP^{(0)})\ge 0$ if $0<\theta<1$, otherwise not required,
		where                    $P^{(0)}=\big[(P_1^{(0)})^{\T},\ldots,(P_v^{(0)})^{\T}\big]^{\T}$;
		
		\STATE $i=0$, and evaluate the objective of (\ref{eq:OMA}) at $\{P_s^{(0)}\}_{s=1}^v$ to $f$;
		
		\REPEAT
		\STATE $f_0=f$, $f=0$;
		\FOR{$s=1$ to $v$}
		\STATE form \eqref{eq:OMA-sub-II''} with either $P_{s'}=P_{s'}^{(i)},\,\,\forall\,s'\ne s$
		for  Jacobi-style updating, or
		$P_{s'}=P_{s'}^{(i+1)},\,\,1\le s'<s$ and $P_{s'}=P_{s'}^{(i)},\,\,s< s'\le v$
		for  Gauss-Seidel-style updating;
		\STATE solve \eqref{eq:OMA-sub-II''} by Algorithm~\ref{alg:SCF} (with $P_s^{(i)}$ as an initial guess) for its maximizer $P_s^{(i+1)}$;
		\ENDFOR
		\STATE $f_0=f$, and evaluate the objective of (\ref{eq:OMA}) at $\{P_s^{(i+1)}\}_{s=1}^v$ to $f$;		
		\STATE $i=i+1$;
		\UNTIL{$|f-f_0|\le\epsilon f$;}
		\RETURN last $\{P_s^{(i)}\in\bbO^{n_s\times k}\}$.
	\end{algorithmic}
\end{algorithm}

Algorithm~\ref{alg:OMA-theta} requires initially
$\tr([P^{(0)}]^{\T}AP^{(0)})\ge 0$. The condition guarantees that the objective (\ref{eq:OMA}) is monotonically
increasing for Gauss-Seidel-style updating if $0<\theta<1$, but it is not necessary if $\theta\in\{0,1\}$, similarly to
what we previously remarked for Algorithm~\ref{alg:SCF}. In case when we don't have an initial guess $P^{(0)}$
satisfying $\tr([P^{(0)}]^{\T}AP^{(0)})\ge 0$, we suggest to set $\theta=0$ and $1$ and iterate until some $P^{(i)}$
such that $\tr([P^{(i)}]^{\T}AP^{(i)})\ge 0$ and then switch back to the original $\theta$. Unfortunately,
it is not clear if the monotonicity property in the objective holds for Jacobi-style updating
even with $\tr([P^{(0)}]^{\T}AP^{(0)})\ge 0$.
In all of our numerical experiments in section~\ref{ssec:egs4MvL}, we simply take $P_s^{(0)}$ to be the first $k$ columns of $I_{n_s}$
and didn't encounter any convergence issue nonetheless for both Jacobi-style and Gauss-Seidel-style updating.

Next we will discuss the convergence of Algorithm~\ref{alg:OMA-theta}.
With Jacobi-style updating, Algorithm~\ref{alg:OMA-theta} generates a sequence $\big\{\{P_s^{(i)}\}_{s=1}^v\big\}_{i=0}^{\infty}$
and the same can be said for with Gauss-Seidel-style updating. But for the convenience of convergence analysis, we
shall expand the sequence by inserting $v-1$ additional intermediate approximations
$$
(P_1^{(i+1)},\ldots,P_s^{(i+1)},P_{s+1}^{(i)},\ldots,P_v^{(i)}),\,s=1,2,\ldots,v-1
$$
into between
$\{P_s^{(i)}\}_{s=1}^v$ and $\{P_s^{(i+1)}\}_{s=1}^v$ in the case of Gauss-Seidel-style updating. We then re-index the expanded
sequence and still denote it by $\big\{\{P_s^{(i)}\}_{s=1}^v\big\}_{i=0}^{\infty}$.

\begin{theorem}\label{thm:OMA-convg1}
	Let  the sequence $\big\{\{P_s^{(i)}\}_{s=1}^v\big\}_{i=0}^{\infty}$  be generated by Algorithm~\ref{alg:OMA-theta},
	and let $\{P_s^{(*)}\}_{s=1}^v$ be an accumulation point of $\big\{\{P_s^{(i)}\}_{s=1}^v\big\}_{i=0}^{\infty}$. Evaluate
	$\what A_{s}$, $\what B_{s}$, and $\what D_{s}$ in \eqref{eq:OMA-KKT} at $\{P_{s'}^{(*)},\,s'\ne s\}$ for each $s$ to
	$\what A_{s}^{(*)}$, $\what B_{s}^{(*)}$, and $\what D_{s}^{(*)}$, respectively.
	\begin{enumerate}[{\rm (a)}]
		\item $(P_s^{(*)}\big)^{\T}\what D_s^{(*)}\succeq 0$ for $1\le s\le v$.
		\item $\{f_{\theta}(\{P_s^{(i)}\})\}_{i=0}^{\infty}$ is monotonically increasing in the case of Gauss-Seidel-style updating and
		thus convergent.
	\end{enumerate}
\end{theorem}

\begin{proof}
	Item (a) holds because $(P_s^{(i)}\big)^{\T}\what D_s\succeq 0$ is designed to hold in Algorithm~\ref{alg:SCF} for
	$\what D_s$ at the time. Because of our expansion in the sequence of approximations in notation for
	Gauss-Seidel-style updating, $\{P_s^{(i)}\}_{s=1}^v$ differs from $\{P_s^{(i+1)}\}_{s=1}^v$ in just one of the $P_s^{(i)}$, and
	that particular $P_s^{(i)}$ is updated by Algorithm~\ref{alg:SCF} so that the objective value is increased. Hence item~(b) holds.
\end{proof}

We introduce a metric for
Cartesian products of $k$ dimension subspaces:
\begin{equation}\label{eq:OMA-sinTheta}
\dist_2(\{\cP_s\},\{\cQ_s\})=\sum_s\|\sin\Theta(\cP_s,\cQ_s)\|_2
\end{equation}
for $(\cP_1,\ldots,\cP_v),\,(\cQ_1,\ldots,\cQ_v)\in\scrG_k(\bbR^{n_1})\times\cdots\scrG_k(\bbR^{n_v})$.
Again the following lemma, similar to Lemma~\ref{lem:isolatedconvg}, is an equivalent restatement of \cite[Lemma 4.10]{moso:1983}
(see also \cite[Proposition 7]{kaqi:1999}) in the context of the product of Grassmann manifolds $\scrG_k(\bbR^{n_s})$.

\begin{lemma}[{\cite[Lemma 4.10]{moso:1983}}]\label{lem:OMA-isolatedconvg}
	Let $\{\cP_s^{(*)}\in \scrG_k(\bbR^{n_s})\}_{s=1}^v$ be an isolated accumulation point
	of the sequence $\big\{\{\cP_s^{(i)}\in\scrG_k(\bbR^{n_s})\}_{s=1}^v\big\}_{i=0}^{\infty}$, in  the metric \eqref{eq:OMA-sinTheta}, such that,
	for every subsequence $\big\{\{\cP_s^{(i)}\}_s\big\}_{i\in\cI}$ converging to $\{\cP_s^{(*)}\}_{s=1}^v$, there is
	an infinite subset $\widehat{\cI}\subseteq \cI$ satisfying
	$\{ \dist_2(\{\cP_s^{(i)}\}_s,\{\cP_s^{(i+1)}\}_s\}_{i\in\what\cI}\rightarrow 0$.
	Then the entire sequence $\big\{\{\cP_s^{(i)}\}_{s=1}^v\big\}_{i=0}^{\infty}$ converges to $\{\cP_s^{(*)}\}_{s=1}^v$.
\end{lemma}

\begin{theorem}\label{thm:OMA-convg2}
	To the conditions of Theorem~\ref{thm:OMA-convg1} add:
	$\{\cR(P_s^{(*)})\}_{s=1}^v$   is an isolated accumulation point
	in the metric \eqref{eq:OMA-sinTheta} of $\big\{\{\cR(P_s^{(i)})\}_{s=1}^v\big\}_{i=0}^{\infty}$ and
	the eigenvalue gaps
	\begin{equation}\nonumber
	\lambda_k(E_s^{(*)}(P_s^{(*)}))-\lambda_{k+1}(E_s^{(*)}(P_s^{(*)}))>0\,\,\mbox{for $1\le s\le v$},
	\end{equation}
	where each $E_s^{(*)}(P_s^{(*)})$ is defined as in \eqref{eq:Es(Ps)} with  $\what A_{s}^{(*)}$, $\what B_{s}^{(*)}$, and $\what D_{s}^{(*)}$.
	Let $r_s:=\rank(\big(P_s^{(*)}\big)^{\T}\what D_s^{(*)})$ for $1\le s\le v$.
	Further,
	suppose\footnote {$\{f_{\theta}(\{P_s^{(i)}\})\}_{i=0}^{\infty}$ is guaranteed convergent for Gauss-Seidel-style updating by
		Theorem~\ref{thm:OMA-convg1}(b).}
	$\{f_{\theta}(\{P_s^{(i)}\})\}_{i=0}^{\infty}$ is convergent for Jacobi-style updating.
	\begin{enumerate}[{\rm (a)}]
		\item The entire sequence $\big\{\{\cR(P_s^{(i)})\}_{s=1}^v\big\}_{i=0}^{\infty}$ converges to $\{\cR(P_s^{(*)})\}_{s=1}^v$.
		\item If also  $r_s=k$ for all $1\le s\le k$, then $\big\{\{P_s^{(i)}\}_{s=1}^v\big\}_{i=0}^{\infty}$
		converges to $\{P_s^{(*)}\}_{s=1}^v$ (in the product of the standard Euclidean metric).
		\item In general, let $\big[P_s^{(*)}\big]^{\T}\what D_s^{(*)}=V_s\Sigma_sV_s^{\T}$ be the singular value decomposition such that the first
		$r_s$ diagonal entries are the nonsingular values, and define
		\begin{multline}\nonumber
		\bbP_s^{(*)}=\left\{P_s^{(*)}(V_s)_{(:,1:r_s)}(V_s)_{(:,1:r_s)}^{\T}\right. \\
		\left.   +P_s^{(*)}(V_s)_{(:,r_s+1:k)}W_s(V_s)_{(:,r_s+1:k)}^{\T}\,:\,W_s\in\bbO^{(k-r_s)\times (k-r_s)}\right\}.
		\end{multline}
		Then $\big\{\{P_s^{(i)}\}_{s=1}^v\big\}_{i=0}^{\infty}$
		converges to the product $\bbP_1^{(*)}\times\cdots\times\bbP_v^{(*)}$ of sets,
		in the sense that
		\begin{equation}\nonumber
		\min_{P_s\in\bbP_s^{(*)}\,\,\forall s} \,\,\sum_s\|P_s^{(i)}-P_s\|_2\to 0\,\,\,\mbox{as $i\to\infty$}.
		\end{equation}
	\end{enumerate}
\end{theorem}

\section{Numerical Experiments}\label{sec:NumExp}
In this section, we will perform two sets of numerical experiments. The first set  demonstrates
the basic behavior of our proposed SCF iteration in Algorithm~\ref{alg:SCF} for problem \eqref{eq:main-op} on synthetic problems,
and  the second set  demonstrates the effectiveness of our multi-view subspace learning model \eqref{eq:OMA} solved by
alternating iterations in Algorithm~\ref{alg:OMA-theta} which uses Algorithm~\ref{alg:SCF} as its computational workhorse.
We compare ours against
the state-of-the-art methods in machine learning for multi-view feature extraction
on five real world data sets. All experiments were conducted by on MATLAB (2018a)
on an Mac laptop using macOS Mojave with Intel Core i9 CPU (2.9 GHz) and 32 GB memory.

\begin{figure}[ht]
	\begin{tabular}{@{}c@{}c@{}c@{}c@{}}
		\includegraphics[width=0.25\textwidth]{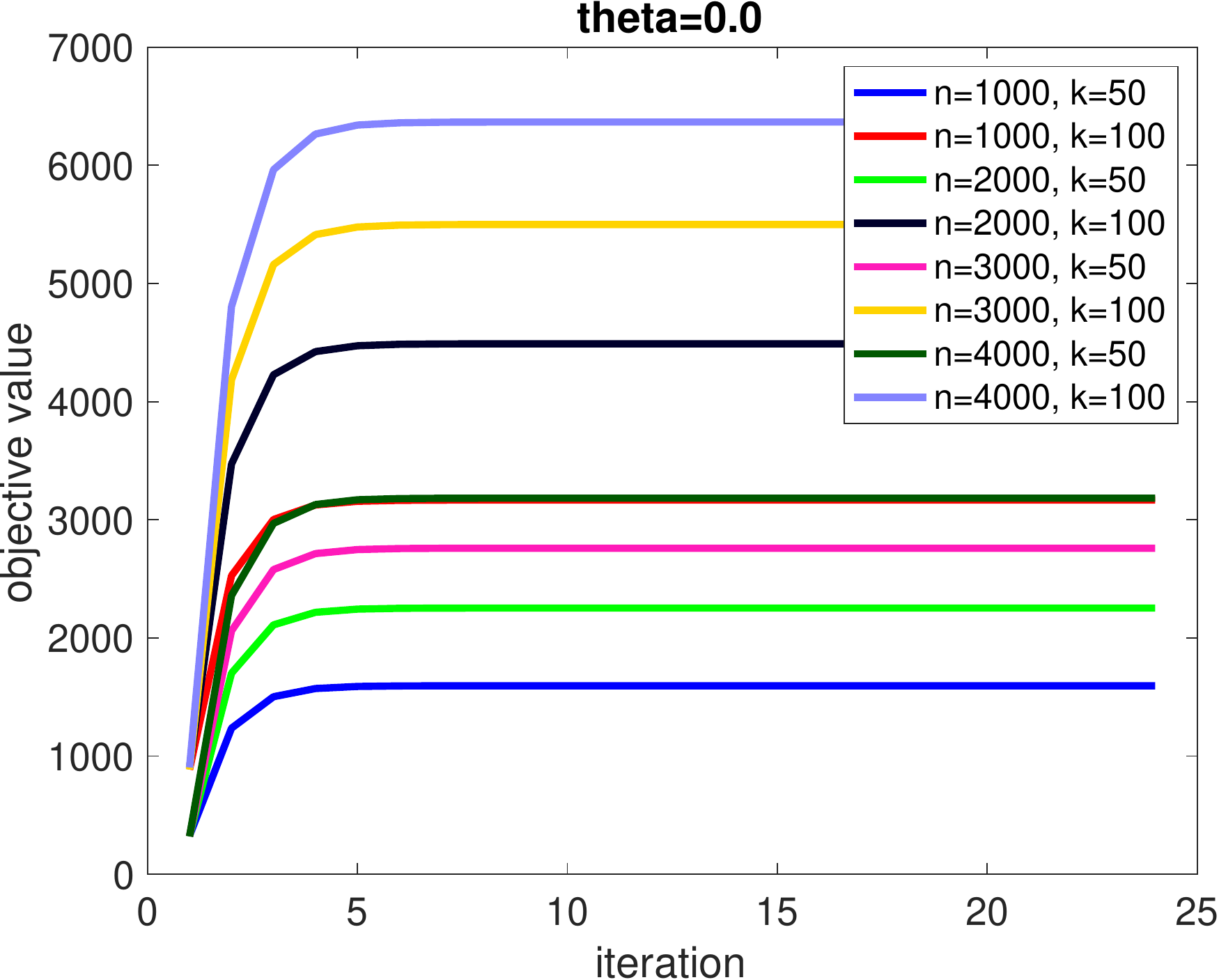} &
		\includegraphics[width=0.25\textwidth]{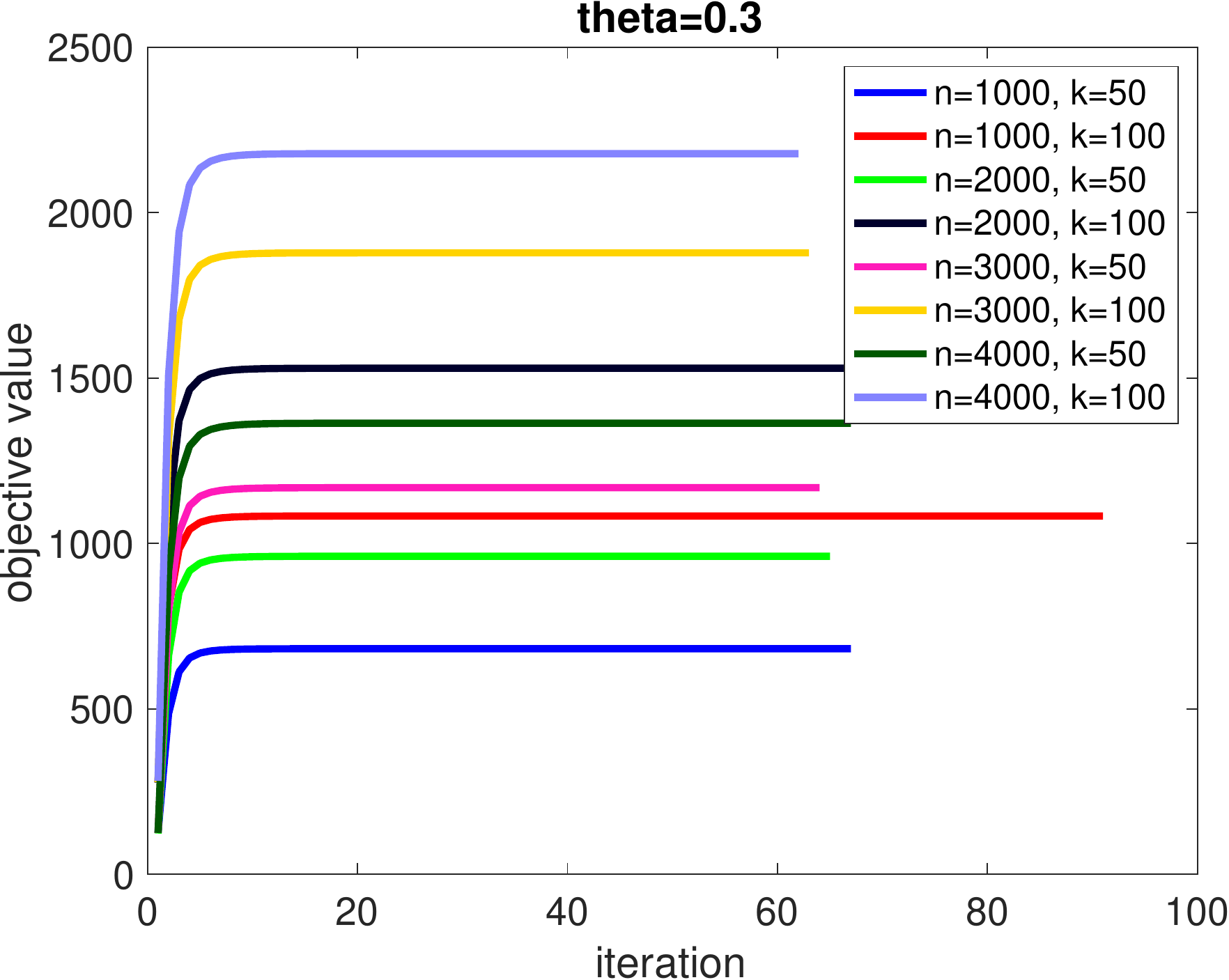}  &
		\includegraphics[width=0.25\textwidth]{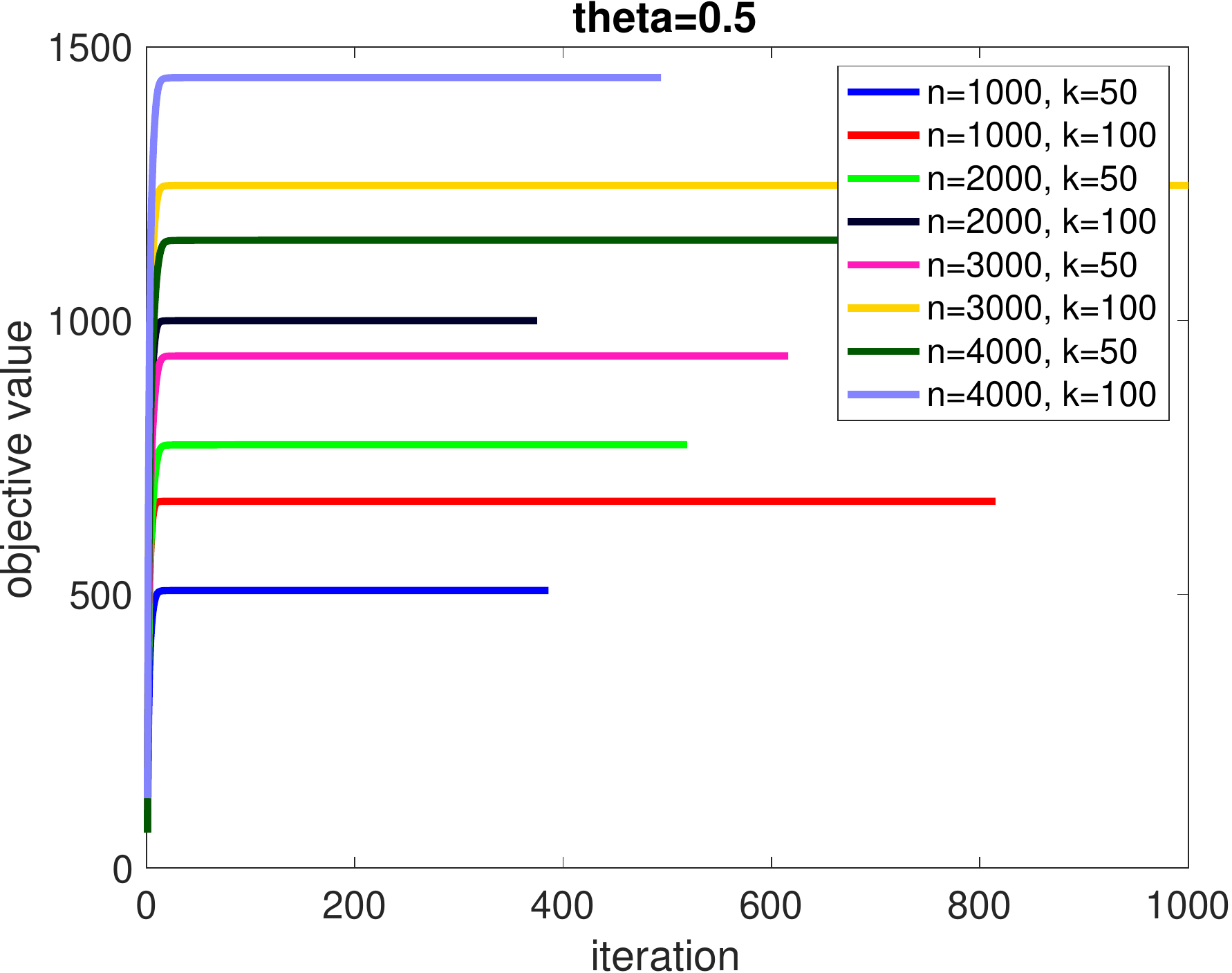}  &
		\includegraphics[width=0.25\textwidth]{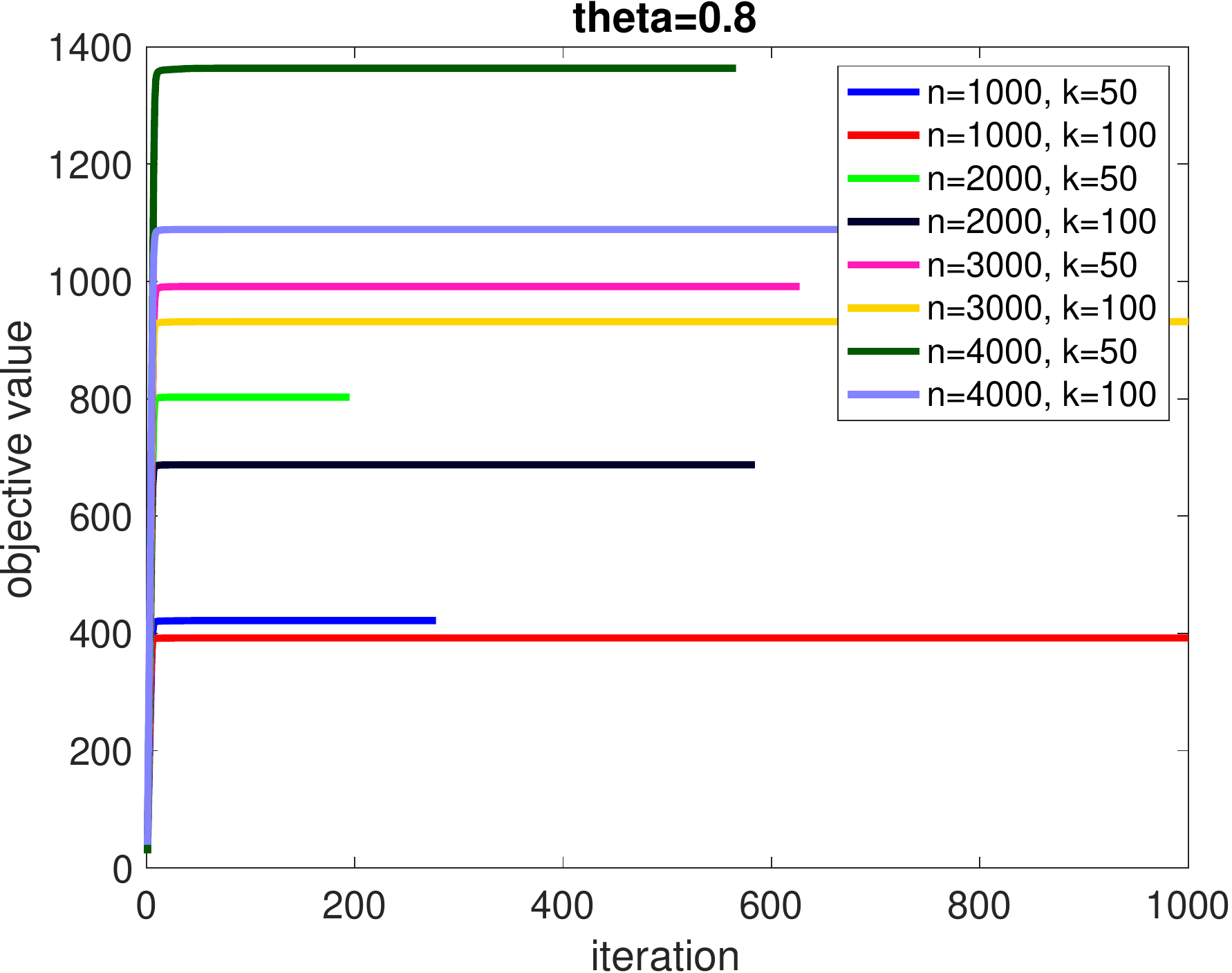}  \\
		\includegraphics[width=0.25\textwidth]{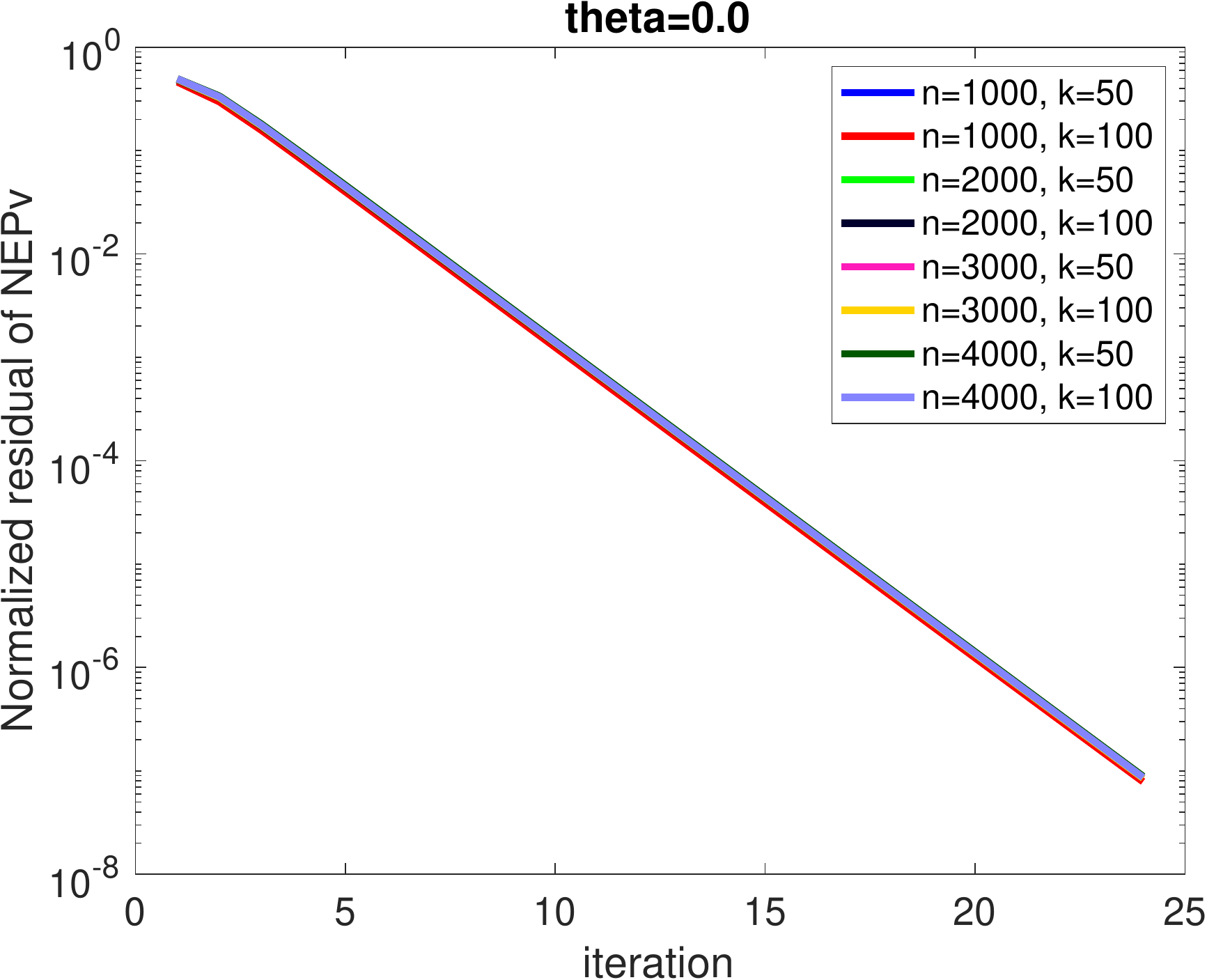} &
		\includegraphics[width=0.25\textwidth]{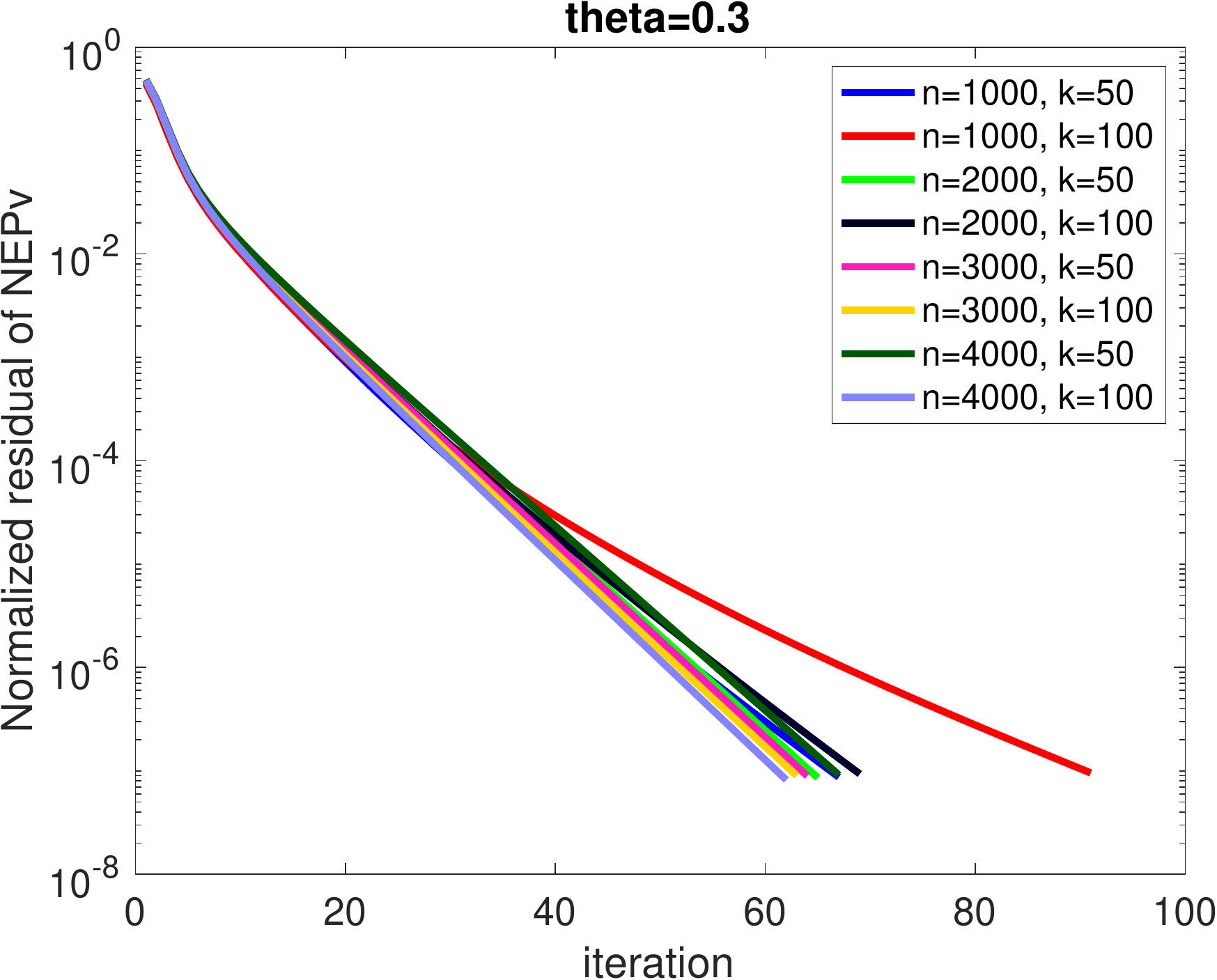}  &
		\includegraphics[width=0.25\textwidth]{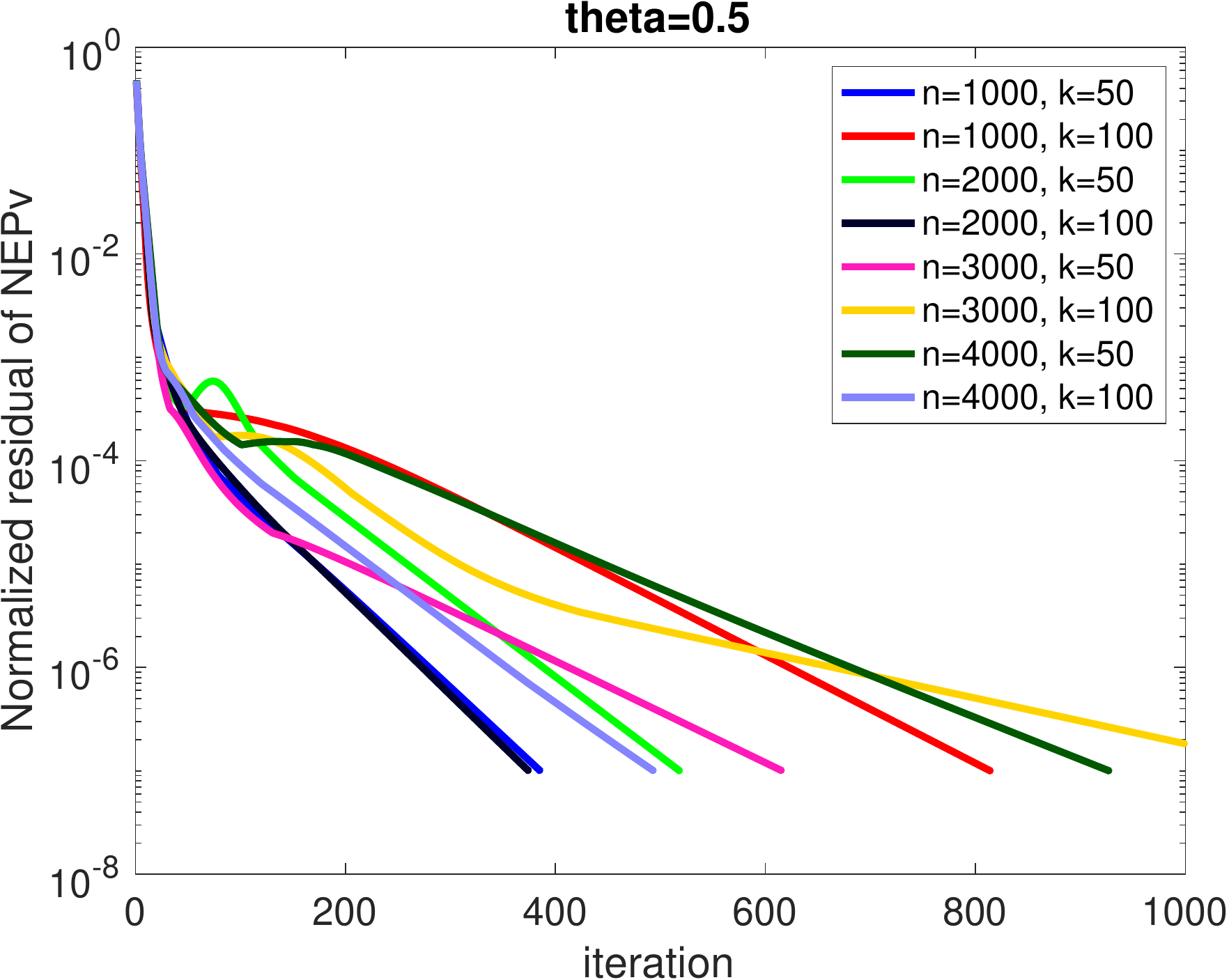}  &
		\includegraphics[width=0.25\textwidth]{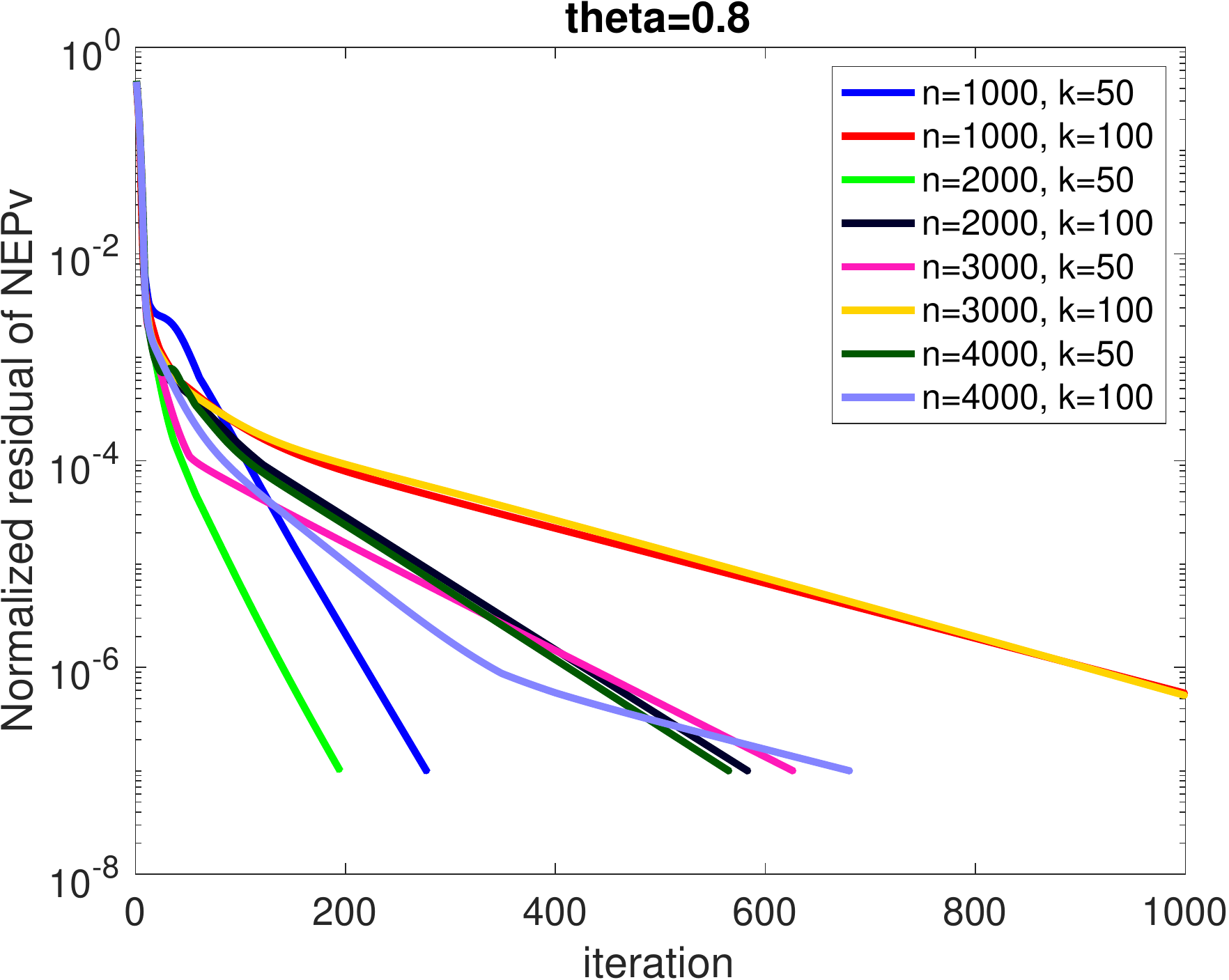}
	\end{tabular}
	\caption{Convergence curves of the objective value and the normalized NEPv residual by Algorithm~\ref{alg:SCF}
		on 8 synthetic problems for $\theta \in \{0, 0.3, 0.5, 0.8\}$. } \label{fig:obj-gnorm}
\end{figure}

\subsection{Experiments on Synthetic Problems}
We first report numerical results on problem \eqref{eq:main-op} solved by our proposed SCF iteration
on synthetic problems, where matrices $A, B$ and $D$ are randomly generated with varying $n \in [1000, 4000]$
and $k \in \{50, 100\}$. Specifically, for a given pair $(n, k)$, matrix $A$ is synthesized by MATLAB
as
\begin{verbatim}
X = randn(n, n); X = (X+X')./2; [U,~] = eig(X);
v = rand(n,1) + 1e-6; A = U * diag(v) * U'; D = randn(n, k);
\end{verbatim}
and $B$ is generated similarly to $A$.  With an increase of $n$ by $1000$ in the given interval,
we generated $8$ synthetic problems in total. Also varying $\theta \in [0,1]$ with an increase by $0.1$,
we tested Algorithm~\ref{alg:SCF} on a total of $88$ problems \eqref{eq:main-op}.
The stopping tolerance ${\tt tol}=10^{-7}$ in \eqref{eq:stop} and the maximum number of iterations is set at $10^3$.
We report convergence curves of both objective function value and the normalized NEPv residual as defined on the left-hand side
of \eqref{eq:stop}.

\begin{figure}[t]
	\centering
	\begin{tabular}{@{}cc@{}}
		\includegraphics[width=0.45\textwidth]{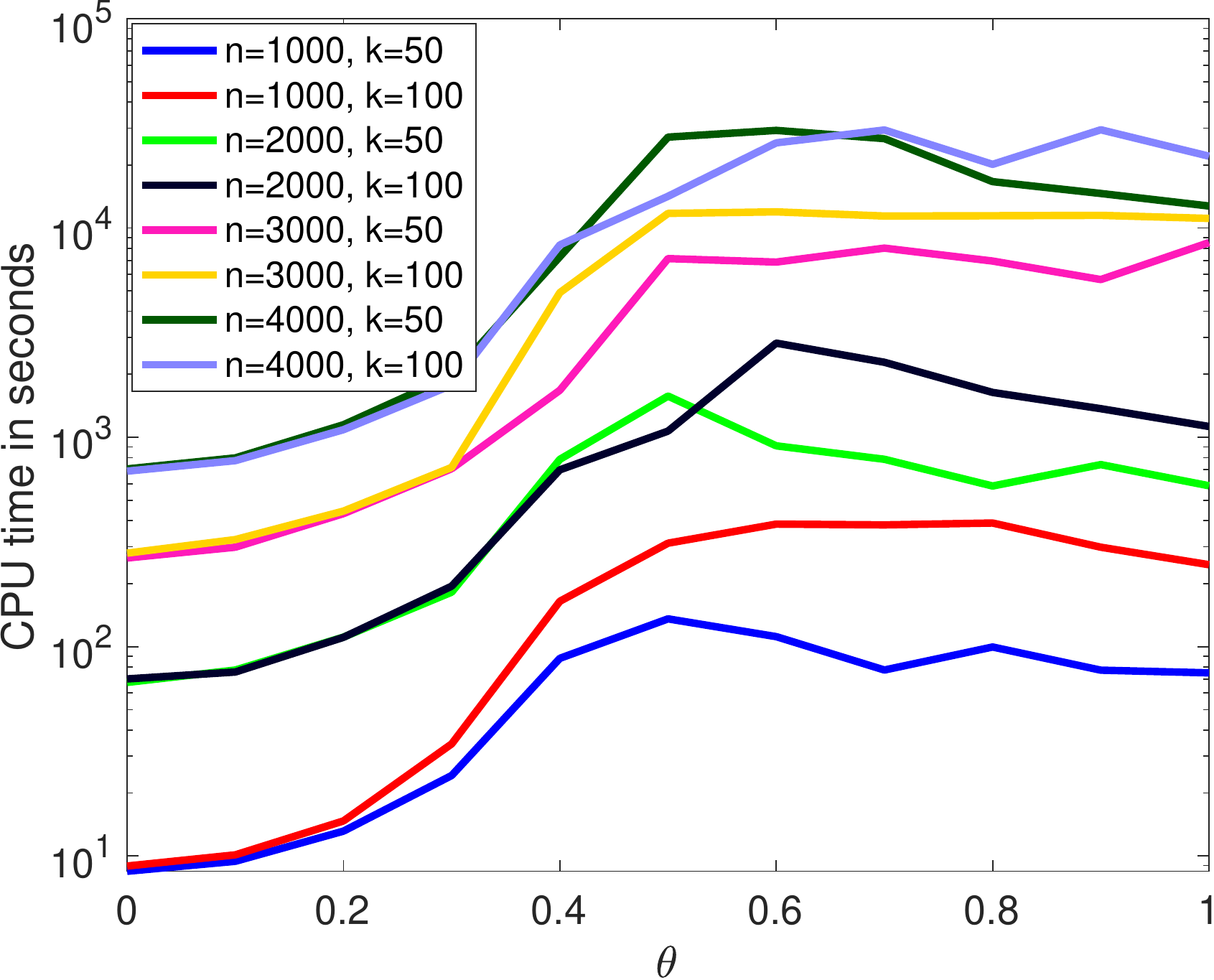} &
		\includegraphics[width=0.45\textwidth]{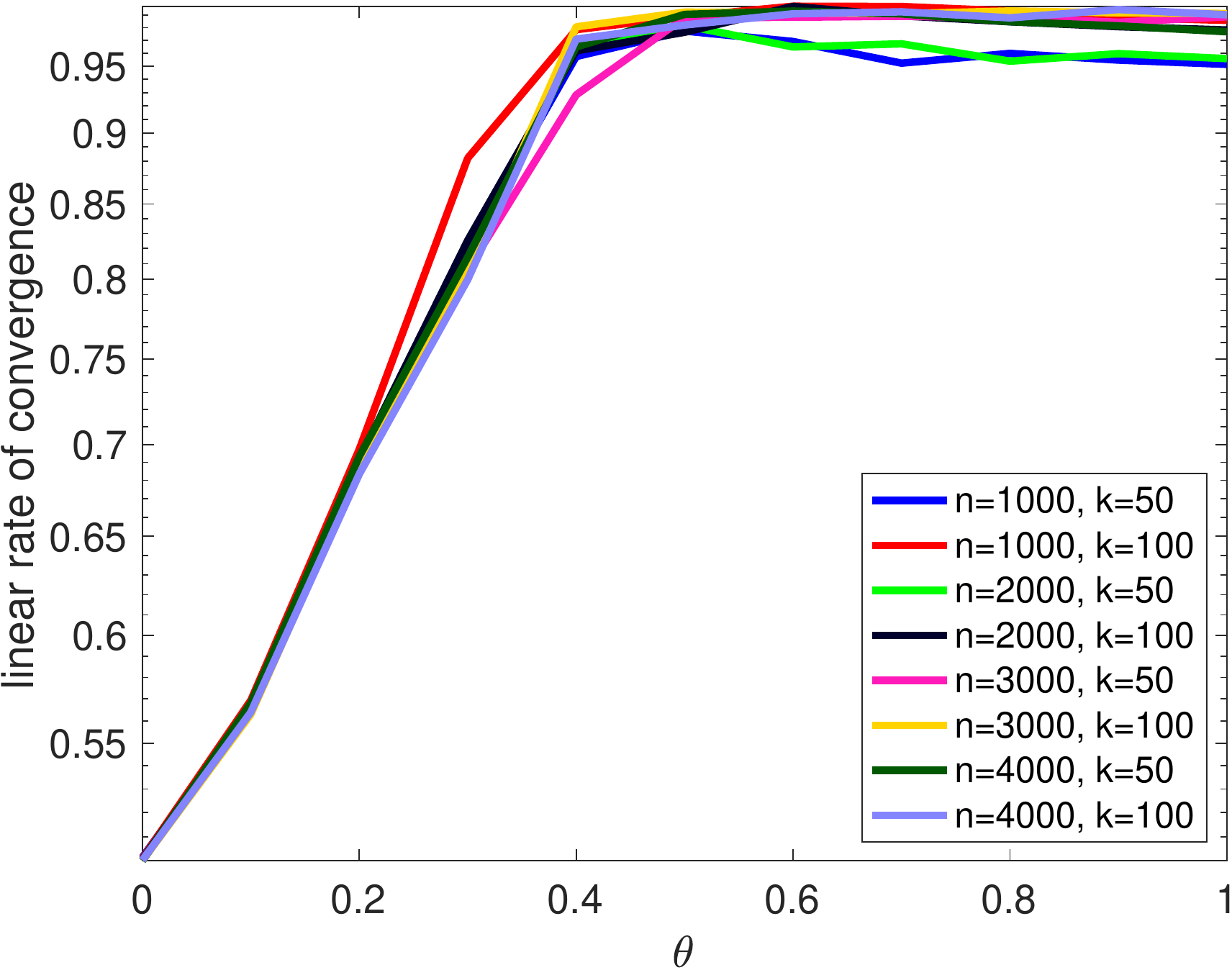}
	\end{tabular}
	\caption{CPU time and the estimated linear convergence rate by  Algorithm~\ref{alg:SCF}
		on 8 synthetic problems for $\theta \in [0,1]$.} \label{fig:cpu}
\end{figure}

Figure~\ref{fig:obj-gnorm} displays the convergence behaviors of Algorithm~\ref{alg:SCF} on $8$ synthetic problems
for selected $\theta \in \{0, 0.3, 0.5, 0.8\}$. As can be observed, most of the curves of the normalized NEPv residual reach the
preset tolerance ${\tt tol}=10^{-7}$ much earlier than the maximum number of iterations.
For these synthetic problems, fewer numbers of iterations are required for smaller $\theta$ than larger ones.
We point out that  the tolerance $10^{-7}$ is often too tiny in machine learning applications. Observe that
in Figure~\ref{fig:obj-gnorm} all objective value curves are very much flat in fewer than 50 SCF iterations.

Figure~\ref{fig:cpu} plots
the CPU times by applying Algorithm~\ref{alg:SCF} on $8$ synthetic problems as $\theta$ varies. These times
are well correlated with the size  $n$. The larger $n$ is, the more CPU time is consumed. For $\theta < 0.2$,
the CPU times are comparable for all problems. As $\theta$ becomes large, more CPU times are consumed for the same $(n,k)$.
This observation is consistent with our estimated rates of linear convergence, which are always under $1$
(demonstrating always convergence) but increase
as $\theta$ does for those synthetic problems (demonstrating more iterations are needed for a larger $\theta$ than a smaller one).
We caution the reader that in general, the rate of linear convergence by Algorithm~\ref{alg:SCF}  is
unlikely  an increasing function of $\theta$ for given $A$, $B$, and $D$.

\subsection{Experiments on Multi-view Data for Feature Extraction}\label{ssec:egs4MvL}
We will specialize the blocks of $A$ and $B$ in \eqref{eq:AB4MvSL} according
to supervised multi-view subspace
learning models including GMA \cite{sharma2012generalized}, MLDA \cite{sun2015multiview} and MvMDA \cite{cao2017generalized}
as detailed in section~\ref{sec:multiview-related}. The resulting concrete models
under the formulation (\ref{eq:OMA}) solved by either Jacobi-style or Gauss-Seidel-style updating scheme
outlined in Algorithm~\ref{alg:OMA-theta} will be named with a suffix ``-J'' or ``-G'' and a prefix ``O'' for ``Orthogonal'' (as
previously in OCCA \cite{zhwb:2020}).
For example, OGMA-J and OGMA-G are (\ref{eq:OMA}) with $A$ and $B$ the same as the choices in GMA and solved by Algorithm~\ref{alg:OMA-theta}
with Jacobi-style and Gauss-Seidel-style updating schemes, respectively.

We evaluate the model (\ref{eq:OMA}) for multi-view feature extraction in machine learning.
Five datasets in Table~\ref{tab:datasets} are used to evaluate the performance of the proposed six concrete models:
OGMA-G, OGMA-J, OMLDA-G, OMLDA-J, OMvMDA-G, and OMvMDA-J, in terms of multi-view feature extraction by comparing them with
their baseline counterparts: GMA, MLDA and MvMDA.
We apply various feature descriptors to extract features of views, including CENTRIST \cite{wu2008place}, GIST \cite{oliva2001modeling},
LBP \cite{ojala2002multiresolution}, histogram of oriented gradient (HOG), color histogram (CH),
and SIFT-SPM \cite{lazebnik2006beyond},
from image datasets: Caltech101 \cite{fei2007learning} and
Scene15 \cite{lazebnik2006beyond}.
Multiple Features (mfeat) and Internet Advertisements (Ads) are publicly available from the UCI machine learning repository \cite{Dua:2019}.
Dataset mfeat  contains handwritten numeral
data with six views
including profile correlations (fac), Fourier coefficients of the character shapes (fou), Karhunen-Love coefficients (kar), morphological features (mor),
pixel averages in $2 \times 3$ windows (pix), and Zernike moments (zer).
Ads is used to predict whether or not a given hyperlink (associated with an image) is an advertisement and has three views: features based on the terms in the images URL, caption, and alt text (url+alt+caption), features based on the terms in the URL of the current site (origurl), and  features based on the terms in the anchor URL (ancurl).

Except for MvMDA and its new variants: OMvMDA-G and OMvMDA-J, all other methods share the same trade-off parameter
$\alpha$ to balance the pairwise correlations and supervised information.
In our experiments, we tune $\alpha \in \{0.01, 0.1, 1, 10, 100\}$ for  proper balancing in supervised setting.
To prevent possible singularity in matrix $B$, we add a small value, e.g., $10^{-8}$, to the diagonals of $B_s\,\forall s$.
For our proposed methods, an additional parameter $\theta$ is varied from $0$ to $1$
with an increase of $0.1$. We also set the maximum number of iterations to $50$ for both the SCF iteration
of Algorithm~\ref{alg:SCF} and the Jacobi-style or Gauss-Seidel-style updating of Algorithm~\ref{alg:OMA-theta}.
It is more of an empirical threshhold observed as a good enough setting for multi-view feature extraction.

\begin{table}[h]
	\caption{\small Real world data sets, where the number of features
		for each view is shown inside the bracket and `-' for  views not applicable.}  \label{tab:datasets}
	\centering
	\begin{footnotesize}
		\begin{tabular}{@{}c|c|c|c|c|c@{}}
			\hline
			& mfeat & Caltech101-7 & Caltech101-20 & Scene15  & Ads \\ \hline
			samples & 2000 & 1474 & 2386 & 4310 & 3279\\
			classes & 10 & 7 & 20 & 15 & 2\\
			view 1 & fac(216) & CENTRIST(254) &   CENTRIST(254) & CENTRIST(254) & url+alt+caption(588) \\
			view 2 &  fou(76) & GIST(512) &GIST(512) & GIST(512) & origurl(495)\\
			view 3 & kar(64) &LBP(1180) & LBP(1180) &LBP(531) & ancurl(472) \\
			view 4 &  mor(6) & HOG(1008) & HOG(1008) & HOG(360) & -\\
			view 5 & pix(240) &CH(64) & CH(64) & SIFT-SPM(1000) &  -\\
			view 6 & zer(47) & SIFT-SPM(1000) &SIFT-SPM(1000) &- &  -\\ \hline
		\end{tabular}
	\end{footnotesize}
\end{table}

To evaluate the classification performance of compared methods, the 1-nearest neighbor classifier as the base classifier is employed.
We run each method to learn projection matrices by varying the dimension of the common subspace $k \in \{2, 3, 5:5:30\}$
for all datasets except $k \in \{2, 3, 4, 5, 6\}$ for  mfeat due to its smallest view {\tt mor} having only $6$ features.
We split each dataset into training and testing with  ratio 10/90.
The learned projection matrices are used to transform both training and testing data into the latent common space, and then
the classifier is trained and tested in this space. Following  \cite{zhwb:2020}, the serial feature fusion strategy
is employed by concatenating projected features from all views.
Classification accuracy is used to measure  learning performance.
Experimental results are reported in terms of the average and standard deviation over 10 randomly drawn splits.

\begin{table}[!ht]
	\setlength{\tabcolsep}{4pt}
	\caption{\small Classification accuracy ($\pm$ standard deviation) of multi-view feature extraction on five real world data sets with 10\% training and 90\% testing over 10 random splits. The best $\theta$ is shown in the bracket.} \label{tab:accuracy}
	\centering
	\begin{footnotesize}
		\begin{tabular}{@{}lccccc@{}}
			\hline
			methods& mfeat& Caltech101-7& Caltech101-20& Scene15& Ads\\
			\hline
			GMA& 93.99 $\pm$ 0.87~~~~~~& 93.25 $\pm$ 1.04~~~~~~& 81.16 $\pm$ 0.94~~~~~~& 61.41 $\pm$ 1.30~~~~~~& 92.59 $\pm$ 1.76~~~~~~\\
			OGMA-J& 96.81 $\pm$ 0.46{\tiny (0.4)}& 95.14 $\pm$ 0.59{\tiny (0.4)}& 86.48 $\pm$ 1.02{\tiny (0.6)}& 79.90 $\pm$ 0.80{\tiny (1.0)}& 94.69 $\pm$ 0.75{\tiny (0.8)}\\
			OGMA-G& 96.80 $\pm$ 0.44{\tiny (0.4)}& 95.07 $\pm$ 0.56{\tiny (0.5)}& 86.60 $\pm$ 1.11{\tiny (0.5)}& 79.90 $\pm$ 1.02{\tiny (1.0)}& 94.91 $\pm$ 0.67{\tiny (0.8)}\\
			\hline
			MLDA& 92.01 $\pm$ 1.74~~~~~~& 92.18 $\pm$ 0.95~~~~~~& 77.79 $\pm$ 1.01~~~~~~& 59.02 $\pm$ 0.94~~~~~~& 92.50 $\pm$ 2.06~~~~~~\\
			OMLDA-J& 96.74 $\pm$ 0.40{\tiny (0.8)}& 94.68 $\pm$ 0.48{\tiny (0.8)}& 86.23 $\pm$ 1.16{\tiny (0.9)}& 81.42 $\pm$ 1.07{\tiny (1.0)}& 94.79 $\pm$ 0.65{\tiny (0.8)}\\
			OMLDA-G& 96.82 $\pm$ 0.38{\tiny (0.8)}& 94.59 $\pm$ 0.62{\tiny (0.3)}& 86.09 $\pm$ 1.22{\tiny (0.9)}& 80.68 $\pm$ 0.88{\tiny (1.0)}& 94.76 $\pm$ 0.74{\tiny (0.8)}\\
			\hline
			MvMDA& 93.78 $\pm$ 0.91~~~~~~& 92.14 $\pm$ 0.68~~~~~~& 79.27 $\pm$ 1.71~~~~~~& 57.33 $\pm$ 1.18~~~~~~& 78.51 $\pm$ 2.96~~~~~~\\
			OMvMDA-J& 96.62 $\pm$ 0.31{\tiny (0.5)}& 95.11 $\pm$ 0.72{\tiny (0.4)}& 85.69 $\pm$ 0.87{\tiny (0.4)}& 77.98 $\pm$ 1.24{\tiny (0.9)}& 94.02 $\pm$ 1.54{\tiny (0.6)}\\
			OMvMDA-G& 96.63 $\pm$ 0.37{\tiny (0.4)}& 94.95 $\pm$ 0.56{\tiny (0.5)}& 85.76 $\pm$ 1.00{\tiny (0.4)}& 78.07 $\pm$ 0.83{\tiny (0.9)}& 93.52 $\pm$ 0.59{\tiny (0.0)}\\
			\hline
		\end{tabular}
	\end{footnotesize}
\end{table}

Table \ref{tab:accuracy} shows the classification accuracies and standard deviations of $9$ multi-view feature extraction methods
on five real data sets over 10 random splits with 10\% training and 90\% testing.
We have observed the following:
\begin{enumerate}[i)]
	\item Our proposed methods consistently outperform their counterparts on all five data sets.
	The least improvement by about 2\% on Ads is observed, while the largest improvement occurs on Scene15 for about 20\%.
	\item The Jacobi-style and Gauss-Seidel-style updating schemes on the same model
	achieve similar classification accuracies with differences falling within 0.8\%.
	\item MvMDA on Ads fails to produce a proper latent representation since its accuracy
	is 14\% less than those of GMA and MLDA. However,  OMvMDA-J and OMvMDA-G perform very well on the same data set.
	\item Accuracies by all 6 proposed models are comparable and very good.
	This may be due to our trace ratio formulation with its orthogonal projection matrices for noisy robustness \cite{zhwb:2020}
	as well as the varying $\theta$ for weighting.
\end{enumerate}

\begin{figure}
	\begin{tabular}{|c|c|c|c|}
		\hline
		& Caltech101-7 & Scene15 & Ads\\ \hline
		\rotatebox[origin=c]{90}{varying $\theta$}&
		\raisebox{-.5\height}{\includegraphics[width=0.29\textwidth]{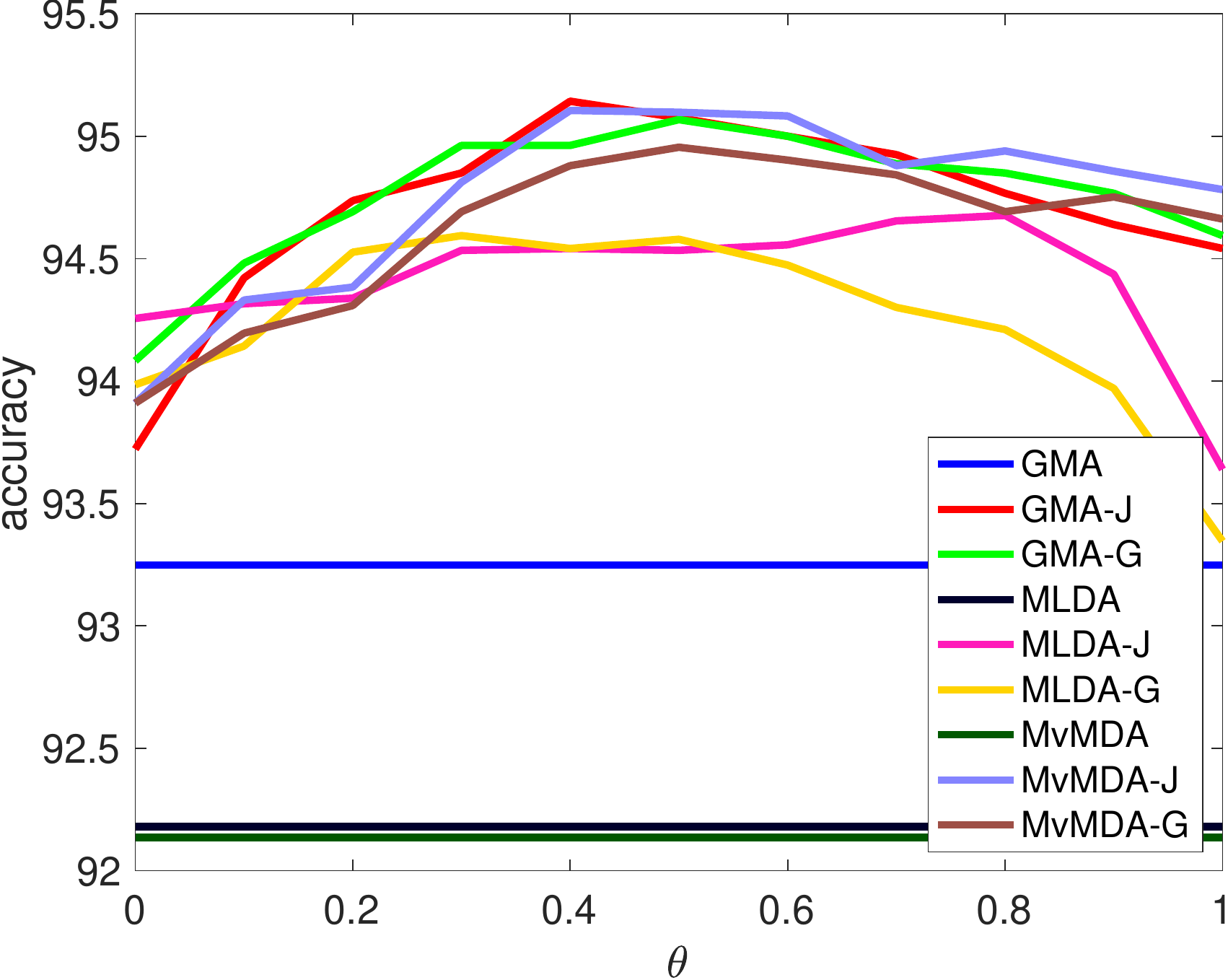}} &
		\raisebox{-.5\height}{\includegraphics[width=0.29\textwidth]{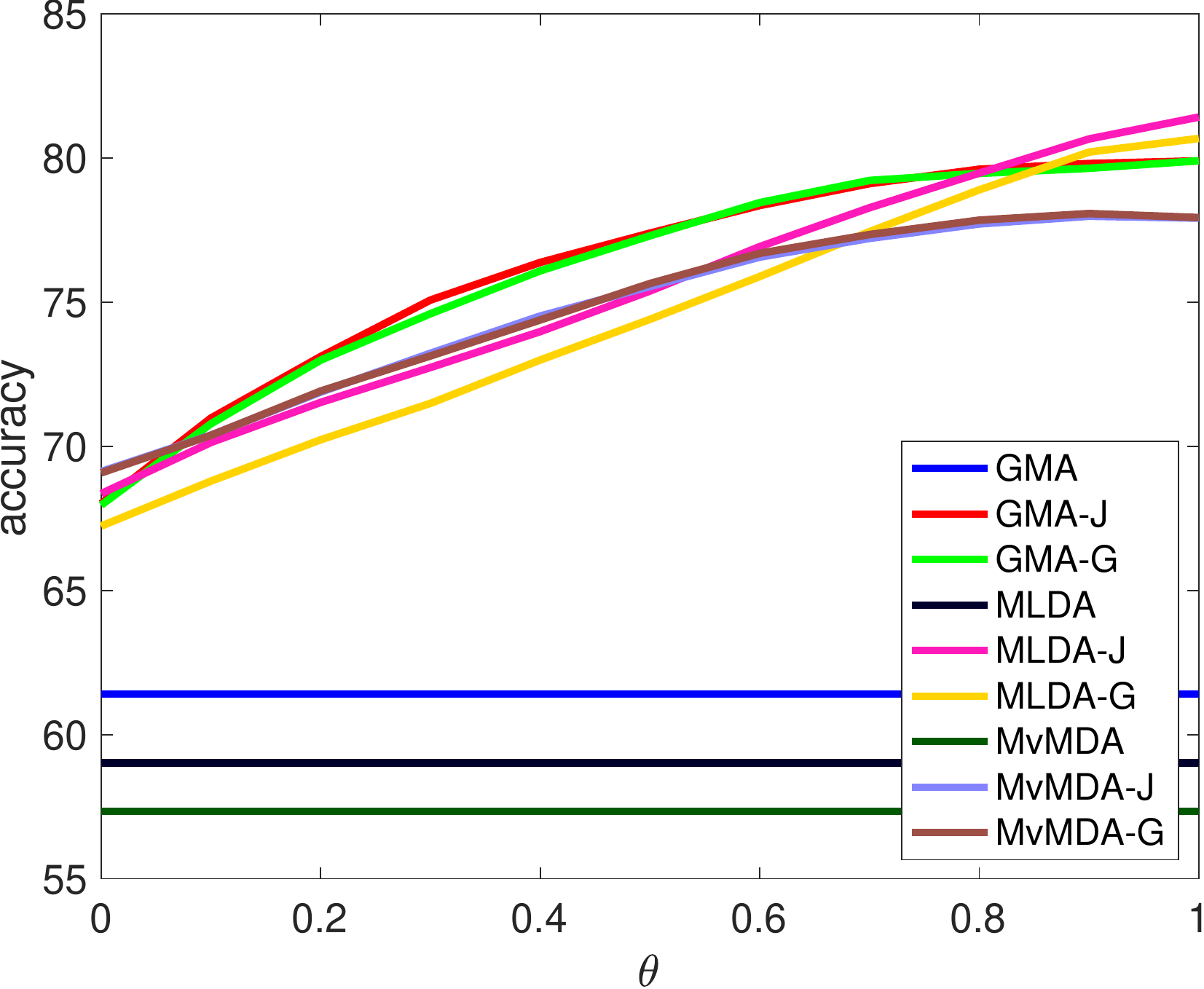}} &
		\raisebox{-.5\height}{\includegraphics[width=0.29\textwidth]{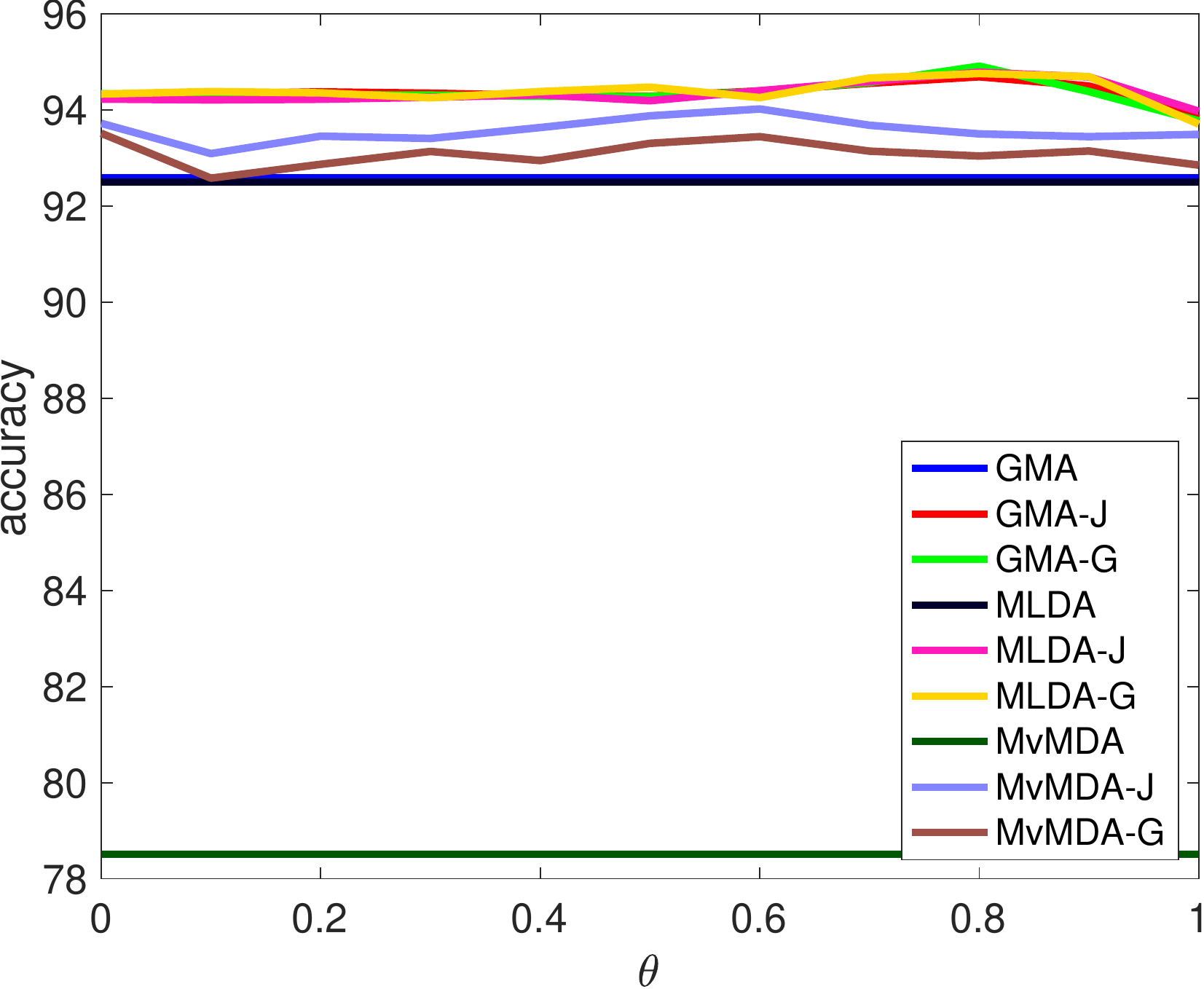}}\\  \hline
		\rotatebox[origin=c]{90}{varying $k$}&
		\raisebox{-.5\height}{\includegraphics[width=0.29\textwidth]{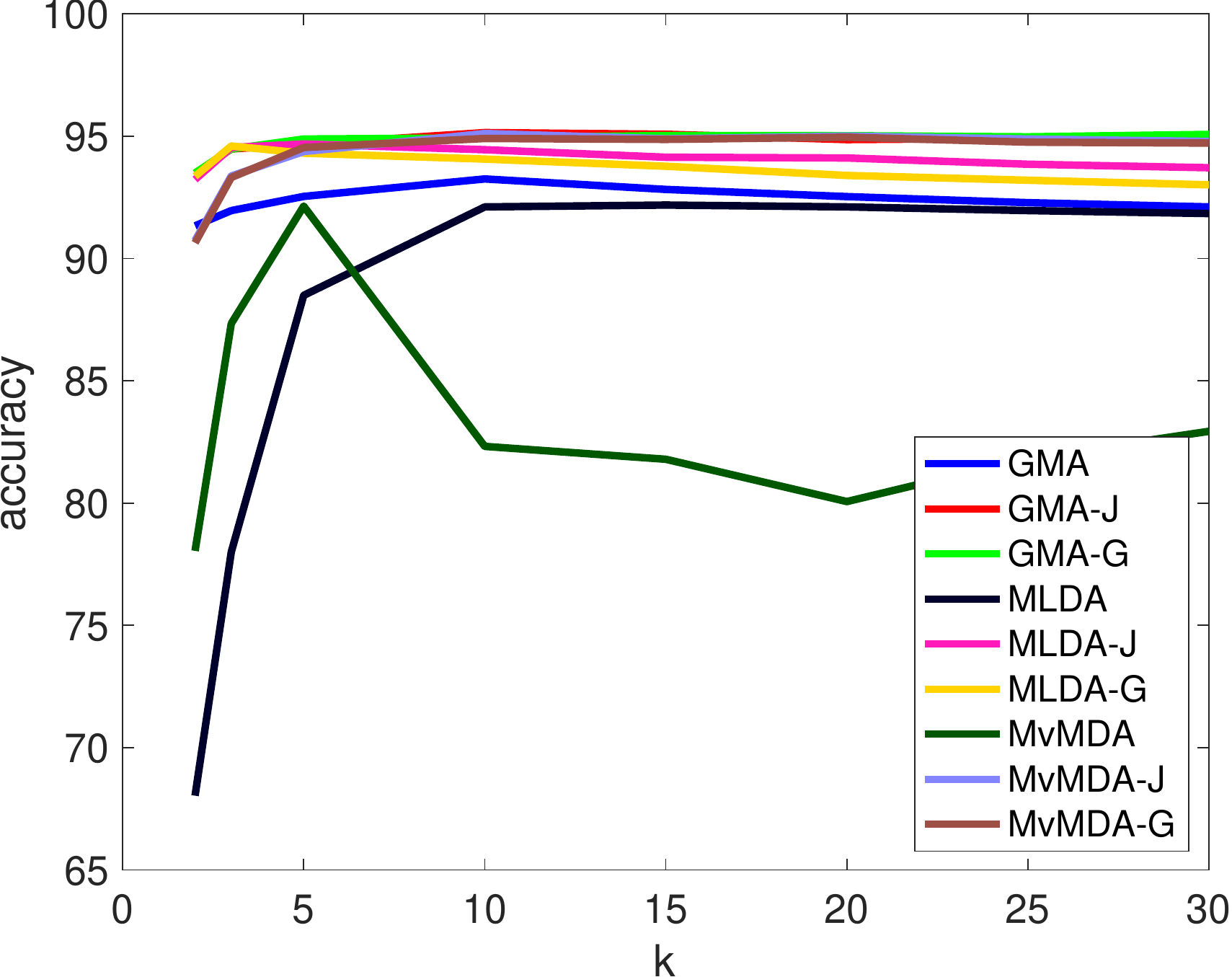}} &
		\raisebox{-.5\height}{\includegraphics[width=0.29\textwidth]{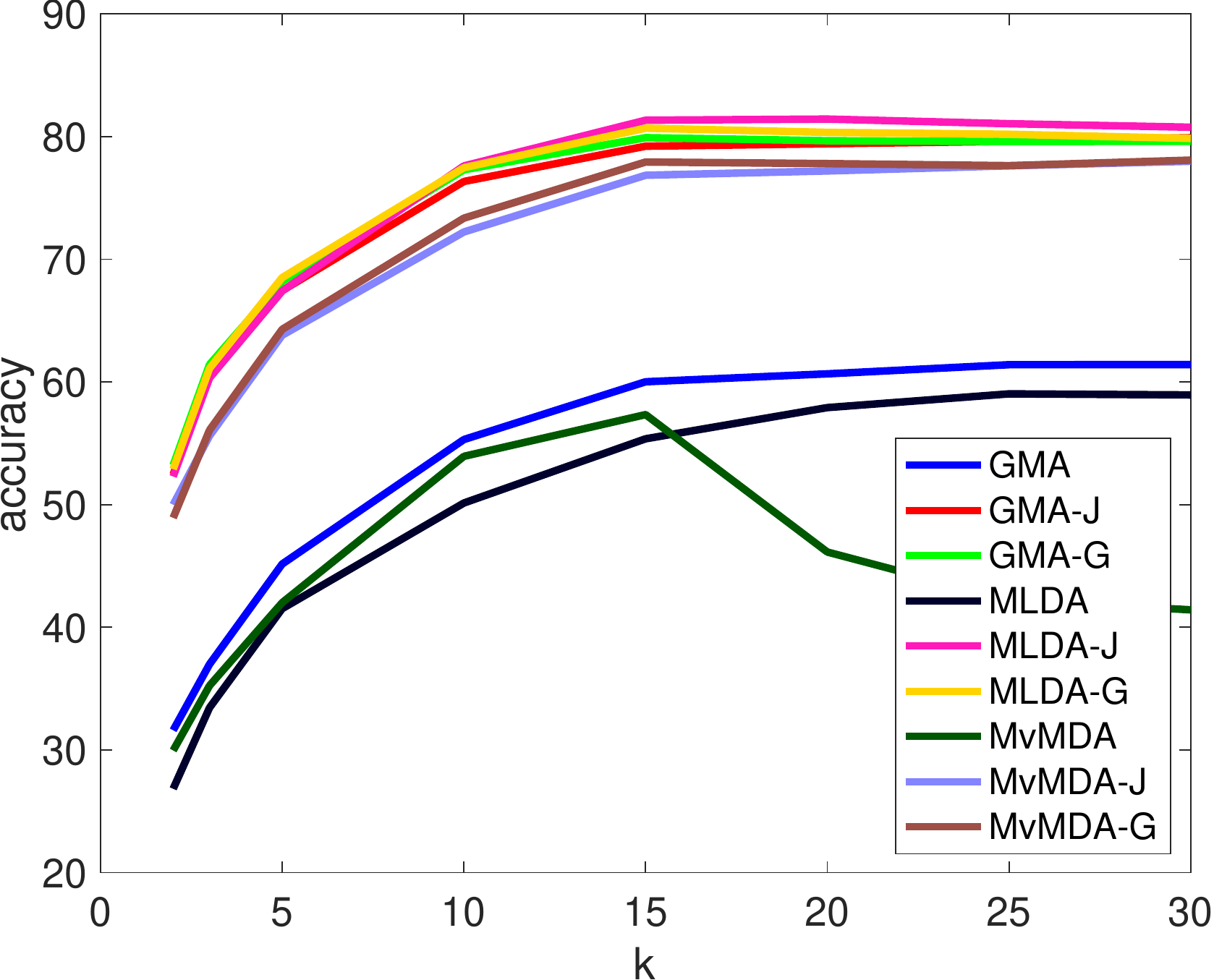}} &
		\raisebox{-.5\height}{\includegraphics[width=0.29\textwidth]{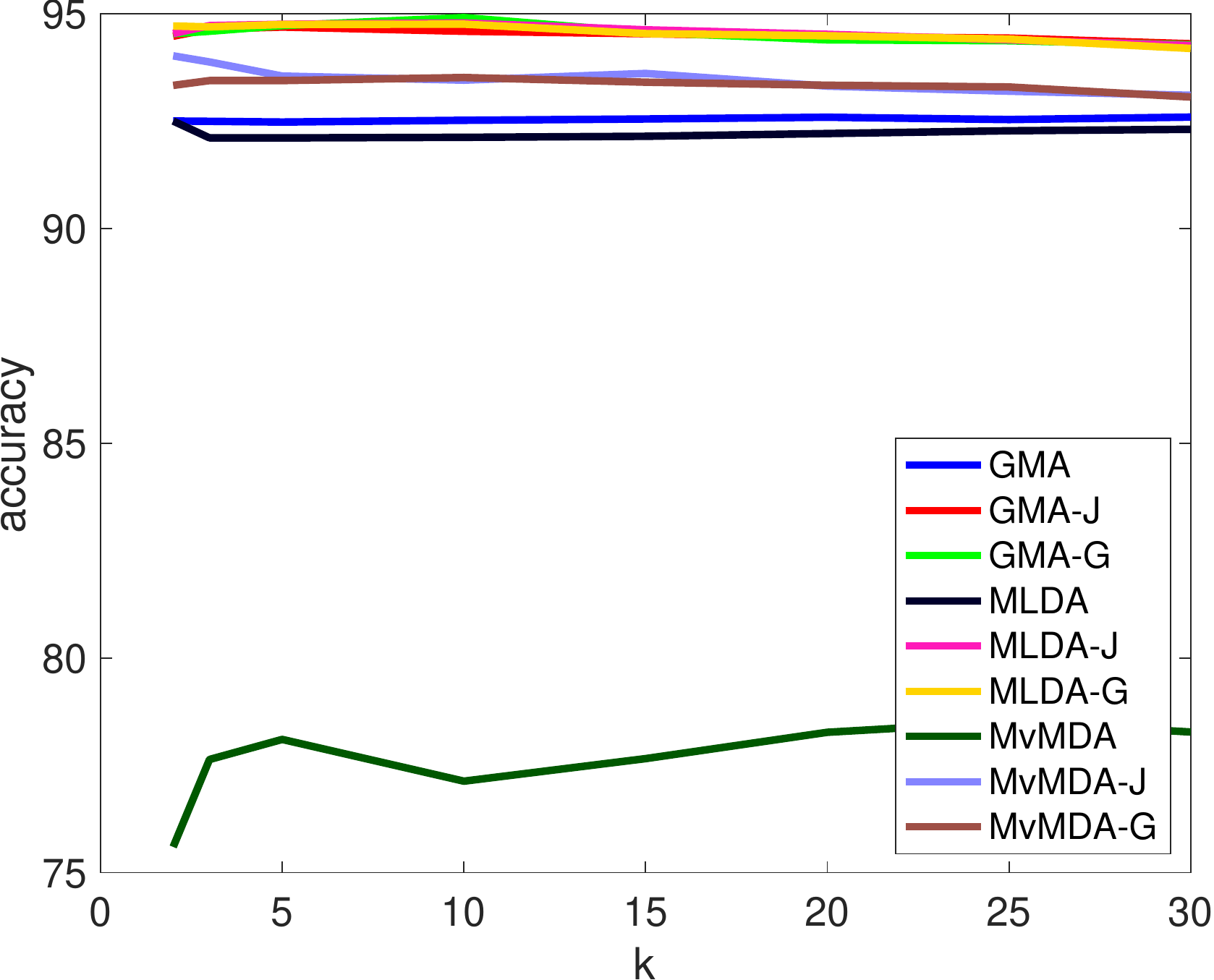}}\\			
		\hline
	\end{tabular}
	\caption{Accuracies of compared methods on three data sets by varying $k \in [2,6]$ and $\theta\in[0,1]$.}
	\label{fig:theta-k}
\end{figure}

In Figure~\ref{fig:theta-k} (the 1st row), we also report the classification accuracy of $9$ methods with varying $\theta \in [0,1]$.
On Caltech101-7, Caltech101-20 and mfeat, the best results of our proposed methods are roughly around $\theta=0.5$.
However, different behaviors are found on Scene15 and Ads. $\theta$ does not show significantly impact on Ads,
but it produces better accuracy on Scene15 as $\theta$ increases. For almost all $\theta$,
our proposed methods consistently outperform baselines.
This implies that $\theta$ introduced in (\ref{eq:OMA}) can be useful to find better projection matrices for multi-view feature extraction.
We further show the trend of classification accuracy of compared methods by varying the dimension of latent common space $k$
in Figure~\ref{fig:theta-k} (the 2nd row). For any fixed $k$, our proposed methods outperform their counterparts.
Importantly, our proposed methods nearly reach their best performances all for fairly small $k$, while baseline methods have to use
larger $k$ to match that. This can be plausibly explained, namely, orthonormal bases retain less redundant information than
non-orthonormal ones.
We also observed that MvMDA behaves  unstably for large $k$ on Caltech101-7 and Scene15
since the accuracy drops too significantly but this does not happen to OMvMDA-G and OMvMDA-J.
In summary, our proposed models not only demonstrate superior performances to baseline methods
but also are more robust to data noise and can achieve the same or better performance at smaller $k$.
Small $k$ implies fast computations if an iterative eigen-solver \cite{bddrv:2000,govl:2013,li:2015} is
used in Algorithm~\ref{alg:SCF} and that is extremely useful for large scale real world applications,
such as cross-modal retrieval \cite{cao2017generalized}, for a fast response time due to less computation costs of
pairwise distances in a lower dimensional space.

\section{Conclusions}\label{sec:concl}
We have completed a thorough investigation, both in theory and numerical solutions, of the trace ratio maximization problem
$$
\max_{X^{\T}X=I_k}\frac {\tr(X^{\T}AX+X^{\T}D)}{[\tr(X^{\T}BX)]^{\theta}}.
$$
At least three special cases of it have been well studied in the past decades because of their immediate applications to
data science. Our main results include an NEPv (nonlinear eigenvalue problem
with eigenvector dependency, a term coined in \cite{cazb:2018}) formulation of its KKT condition, necessary conditions for its
local and global maximizers, a complete picture of the role played by $D$ on the maximizers, a guaranteed convergent numerical method
and its full convergent analysis. As an application of these results, we propose a new orthogonal multi-view subspace framework
and experiment on its 6 instantiated models in either supervised or unsupervised setting. Numerical results demonstrate the new models
often outperform existing baselines.

Although we have been limiting our discussion on real matrices, the developments in this paper can be straightforwardly extended to
complex matrices with minor modifications, namely, replace all $\bbR$ by $\bbC$ (complex numbers) and all transposes $\T$ by complex conjugate
transpose $\HH$.

\end{document}